\documentclass[12pt,a4paper]{article}

\usepackage{geometry}
\usepackage{titling}
\usepackage{amsopn, amstext, amsmath,amsthm, amsbsy,mathrsfs}
\usepackage{latexsym,amssymb}
\usepackage{ragged2e}
\usepackage{url}
\usepackage{bbm}
\usepackage{authblk}
\usepackage{amsfonts}
\usepackage{dsfont}
\usepackage{color}
\usepackage{booktabs}
\usepackage{multirow}
\usepackage{array}

\usepackage{diagbox, fancyvrb}
\usepackage[english]{babel}

\usepackage{setspace} 
 
\usepackage{titlesec} 
\usepackage{titling}
\titleformat*{\section}{\normalsize\bfseries}
\titleformat*{\subsection}{\normalsize\it\bfseries}
\titlespacing*{\section}{0pt}{6pt}{6pt}

\titlelabel{\thetitle.\quad}

\RequirePackage{bm}

\numberwithin{equation}{section}
\usepackage[square,numbers,sort&compress]{natbib}

\usepackage[colorlinks,
linkcolor=blue,
anchorcolor=blue,
citecolor=blue
]{hyperref}

\newcommand{\sref}[2]{\hyperref[#2]{#1 \ref{#2}}}
\newcommand{\myeqref}[1]{\hyperref[#1]{\eqref{#1} }}

\newcommand\keywords[1]{\text{keywords}: #1 }
\newcommand{\as}{$\operatorname{\mathbb{P}-a.s.}$}
\newtheorem{theorem}{Theorem}[section]  
\newtheorem{Corollary}{Corollary}[section]    
\newtheorem{lemma}{Lemma}[section]  
\newtheorem{proposition}{Proposition}[section]  
\newtheorem{example}{Example}[section]  
\newtheorem{definition}{Definition}[section]  
\newtheorem{remark}{Remark}[section]  

\date{\empty}
\renewenvironment{abstract}{\par\noindent\textbf{\abstractname}\ \ignorespaces}{\par\medskip}

\begin{document}

	\title{\large Quadratic BSDEs with Singular Generators and Unbounded Terminal Conditions: Theory and Applications}
			\author[a]{Wenbo Wang} 
		\author[a,b,$ \ast $]{Guangyan Jia}
		\affil[a]{\normalsize Zhongtai Securities Institute for Financial Studies,  Shandong University, Jinan, P.R. China}
		\affil[b]{\normalsize Shandong Province Key Laboratory of Financial Risk, Jinan, P.R. China}
		
		\vspace{0.8cm}
	\maketitle 		\thispagestyle{empty}
		\begin{spacing}{1.25}
		\begin{flushleft}
		{\small	
			\begin{abstract}
				
				\justifying \noindent
				We investigate a class of quadratic backward stochastic differential equations (BSDEs) with generators singular in $ y $. First, we establish the existence of solutions and a comparison theorem, thereby extending results in the literature. Additionally, we  analyze the stability property and the Feynman-Kac formula, and prove the uniqueness of viscosity solutions for the corresponding singular semilinear partial differential equations (PDEs). Finally, we demonstrate applications in the context of robust control linked to stochastic differential utility and certainty equivalent  based on $g$-expectation. In these applications, the coefficient of the quadratic term in the generator captures the level of ambiguity aversion and the coefficient of absolute risk aversion, respectively.
				
			\end{abstract}

\keywords{quadratic backward stochastic differential equation; singular generators; unbounded terminal conditions; viscosity solution; comparison theorem; stochastic differential utility} 
}
\end{flushleft}
\end{spacing}

\section{Introduction}

This paper focuses on a one-dimensional  backward stochastic differential equation (BSDE)
	\begin{equation}\label{eq1}
	Y_{t} = \xi + \int_{t}^{T}f(s,Y_{u},Z_{u})  du - \int_{t}^{T}Z_{u} dW_{u}, \quad t \in [0,\, T],
	\end{equation}
where $W$ denotes a standard $ d $-dimensional Brownian motion defined on a filtered probability space
$(\Omega,\, \mathcal{F},\, (\mathcal{F}_{s})_{s \in [0, \, T]}, \mathbb{P})$. Here, $(\mathcal{F}_{s})_{s \in [0, \, T]}$ is the $\mathbb{P}$-completion of the filtration generated by $W$. The terminal condition $ \xi $ is an $ \mathcal{F}_{T}$-measurable $ \mathbb{R} $-valued random variable and the generator $f: [0, T] \times \Omega \times \mathbb{R} \times \mathbb{R}^{1\times d} \mapsto \mathbb{R} $ is a progressively measurable process. Nonlinear BSDEs are pioneered under the Lipschitz condition by \citep{Peng1990}.  
BSDEs with generators that exhibit quadratic growth in regards to the variable $ z $ have been extensively researched by articles \citep{BriandHu2006, BriandHu2008, Tevzadze2008, 2013Monotone, BriandRichou2017} and others since \citep{Kobylanski2000} first investigated the situation of  bounded terminal conditions. This paper characterizes a specific category of BSDEs with the form of generators being $g(y)|z|^{2}$, which first appeared in \citep{DuffieEpstein1992} to our knowledge. 
Through employing It\^o-Krylov's formula as well as a ``domination method''(see \sref{Lemma}{lemBahlali}) derived from the existence result for reflected BSDEs obtained in \citep{Essaky2011}, 
\citep{Bahlali2017} demonstrated existence and uniqueness results when $g$ is globally integrable on $\mathbb{R}$. Another intriguing case is $f(y,\, z)=|z|^{2}/y$, that is to say, the generator $ f $ is singular at $ y=0 $. In this instance, \citep{Bahlali2018} assume that the generator $ f $ satisfies

\begin{equation}\label{eq2}
0 \leq f(s,\,\omega,\,y,\, z) \leq a_{s}+b_{s} y+\gamma_{s} z+\frac{\delta}{2y}|z|^{2},\, (s,\,\omega,\,y,\, z) \in [0,\, T]\times \Omega \times (0,\, +\infty) \times  \mathbb{R}^{1 \times d},
\end{equation}
for some processes $ a,\, b,\, \gamma $ and constant $ \delta $. 	
By utilizing the domination method, they were able to establish the existence of solutions in $ \mathcal{S}^{p} \times \mathcal{L}^{2} $. Furthermore, they proved a uniqueness result of bounded solutions using convex duality techniques. 
In this paper, we prove that solutions exist in $\mathcal{S}^{p} \times \mathcal{M}^{q}$, and establish the comparison theorem for $ L^{p} $ solutions through the $\theta $-techniques (see \citep{BriandHu2008}). Moreover, our approach allows for straightforward extension of these results to more generalized BSDEs with generators satisfying
$$ 0 \le f(s,\,\omega,\,y,\,z) \le a_{s}+b_{s}\phi(y)+\gamma_{s}|z|+\frac{|z|^{2}}{2\psi(y)} ,\,(s,\,\omega,\, y,\, z) \in [0,\, T] \times \Omega \times (0,\, +\infty) \times \mathbb{R}^{1 \times d},   $$
under appropriate assumptions. Moreover, we derive the stability property for $L^{p}$ solutions and introduce the Feynman-Kac formula within our theoretical framework. Additionally, we investigate the uniqueness of viscosity solutions to the related singular semilinear quadratic PDEs.

More importantly, this class of quadratic BSDEs with singularities and related partial differential equations (PDEs) have significant applications in physics, quantitative finance, biology and so on (see \citep{2012Giachetti, Shamarova2024, Vazquez2007, Bahlali2018}). For instance in finance, the conditional generalized entropic risk measure (see \citep{2021Generalized}) could be expressed in $ Y $, the value solution of a solution to such a BSDE. 
\citep{Bahlali2018} illustrated that $ Y $ can
be ordinally equivalent to the stochastic differential utility of Epstein-Zin type with relative risk aversion represented by $  0 \le \delta \neq 1$ in \myeqref{eq2}. In \sref{Subsection}{subsec5}, we shall demonstrate that $ Y $ can represent a stochastic differential utility concerning robustness while $ 1/\psi(y)$, the coefficient of $ |z|^{2}/2$, quantifies the level of ambiguity aversion. In \sref{Example}{ex2}, we show that $ Y $ can depict the certainty equivalent of the terminal value based on g-expectation while $ 1/\psi(y)$ represents the coefficient of absolute risk aversion. 
 
 	The structure of this paper is as follows. \sref{Sections}{sec2} and \ref{sec3} introduce the notations, present the existence results, and establish the comparison theorem. In \sref{Section}{sec4}, we prove a stability result for the $ L^{p} $ solutions, derive the Feynman-Kac formula within our theoretical framework, and verify the uniqueness of viscosity solutions to the associated PDE. Finally, \sref{Section}{sec5} presents an application linking robust control to stochastic differential utility, which slightly generalizes the results of \citep{Skiadas2003} and \citep{Maenhout2004} and includes an example of certainty equivalent based on g-expectation.

\section{Notations and Existence Results}\label{sec2}

	This section establishes the existence results for a class of quadratic BSDEs with singular generators. Firstly, we need to introduce the notations used in this paper. We say a process or random variable satisfies some property if this holds except on the empty subset. Thus we sometimes omit \as. $ \mathbb{E}_{t} \left[  \cdot \right] := \mathbb{E}\left[  \cdot \middle | \mathcal{F}_{t} \right] $ denotes the conditional mathematical expectation with respect to $ \mathcal{F}_{t} $. The set composed of stopping times $ \tau $ satisfying $ 0\le \tau \le T $ is represented by  $ \mathcal{T}_{0,T}  $. If $ N $ is an adapted and c\`adl\`ag process, define $ N^{\ast} := \sup_{t \in [0,T]} \lvert N_{t} \rvert $. Let us recall that class $ (D) $ consists of progressively measurable processes $ X $ satisfying that $\{ X_{\tau}: \tau \in \mathcal{T}_{0,\,T} \}$ is uniformly integrable. Sometimes we denote stochastic integral $ \int_{0}^{\cdot} Z_{s} dW_{s} $ by $ Z\circ W $ and denote by $ \mathscr{E}(X) $  the stochastic exponential of a one-dimensional local martingale $ X $. Given $p\geq 1$ and open sets $ U \subset \mathbb{R},\,V \subset \mathbb{R}^{n} $, let us define the following spaces and notations.

$\mathcal{C}^{p}(U)$: the space of functions from $U $ to $ \mathbb{R}$ having continuous p-th derivative.

$ L^{p}_{loc}(U)$:  the space of locally $ L^{p} $ integrable functions from $U $ to $ \mathbb{R}$.

$\mathcal{W}^{2}_{p, loc}(U) $: the Sobolev space of functions $ h :U \mapsto \mathbb{R}$ fulfilling $ h $ together with its generalized derivatives $ h' $ and $ h'' $ belong to $ L^{p}_{loc}(U) $.

$\mathcal{S}:=\mathcal{S}(\mathbb{R}) $, the space of adapted and continuous processes valued in $ \mathbb{R} $.

$ \mathbb{L}^{p}:= \mathbb{L}^{p}(\Omega, \mathcal{F}_{T}, \mathbb{P}; \mathbb{R}) $, the space of random variables $\eta$ fulfilling $ \eta$ is valued in $ \mathbb{R},\,\mathbb{E}\left[\lvert \eta\rvert ^{p}\right] < +\infty $, and $ \eta$ is $ \mathcal{F}_{T} $-measurable.

$\mathcal{S}^{\infty}:= \mathcal{S}^{\infty}(\mathbb{R}) $, the space of bounded processes in $  \mathcal{S}$.

$\mathcal{S}^{p}:= \mathcal{S}^{p}(\mathbb{R}) $, the space of all processes $Y\in \mathcal{S}$ such that $\lVert Y \rVert_{\mathcal{S}^{p}}:=\left(  \mathbb{E}[\sup\limits_{0 \le t \le T}\lvert Y_{t}\rvert ^{p}]\right)^{\frac{1}{p}} < + \infty $.

$\mathcal{L}^{2}:=\mathcal{L}^{2}(\mathbb{R}^{1\times d})$, the space of $\mathbb{R}^{1\times d}$-valued processes $N$ satisfying $N$ is progressively measurable and a.s. $ \int^{T}_{0}|N_{s}|^{2}ds < \infty $.

$\mathcal{M}^{p}:= \mathcal{M}^{p}(\mathbb{R}^{1\times d})$, the space of progressively measurable processes $G$ fulfilling that $G$ is valued in $\mathbb{R}^{1\times d}$ and $\lVert G \rVert_{\mathcal{M}^{p}}:=\left(  \mathbb{E}\left[\left(\int^{T}_{0}\lvert G_{s}\rvert ^{2}ds\right)^{p/2}\right]\right)^{1/p} < + \infty $.

$  \mathcal{K} $: the set of $\mathbb{R}$-valued processes $K$ satisfying $K_{0}=0 $, $K$ is continuous, nondecreasing, progressively measurable, and $K_{T} <+\infty$ a.s..

$BMO$: the space of all uniformly integrable martingales $ N $ satisfying $$ \lVert N \rVert_{BMO}: = \sup\limits_{\tau \in \mathcal{T}_{0,T }}\big\lVert \mathbb{E}\Big[ \lvert  N_{T}-N_{\tau}\rvert  \big| \mathcal{F}_{\tau} \Big]\big\rVert_{\infty} \! < \infty .$$

$  D^{2,+}u $: for any function $ u: V \mapsto \mathbb{R},\,D^{2,+}u$
maps $ V $ into the set of subsets of $ V\times \mathbb{R}^{n} \times S(n)  $, where $ S(n) $ is the set of real symmetric $ n \times n $ metrices. For $ \hat{x} \in V$, we define $ (u(\hat{x}),\,p,\, K) \in D^{2,+}u(\hat{x}) $ to mean the following inequality holds
$$ u(x) \le u(\hat{x}) + p^{T}(x-\hat{x}) + \frac{1}{2}(x-\hat{x})^{T}K(x-\hat{x})+o(|x-\hat{x}|^{2}) \text{ as } V \ni x \to \hat{x}.$$

$\bar{D}^{2,+}u $ denotes the closure of $ D^{2,+}u$. That's to say, $ (s,\,p,\,K) \in \bar{D}^{2,+}u(x)$ if there exist
$$ V \ni x_{n}\to x \text{ and } (u(x_{n}),\,p_{n},\,K_{n}) \in D^{2,+}u(x_{n})$$
such that $ (u(x_{n}),\,p_{n},\,K_{n}) \to (s,\,p,\,K) $(see more about viscosity solutions in \citep{Crandall1997}).

We sometimes describe the BSDE with generator $ f $ and terminal $ \xi $ as BSDE$ (\xi, f) $ rather than \myeqref{eq1} for notational simplicity.
A solution of BSDE$ (\xi, f) $ is defined as a pair of processes $(Y,\,Z):=\{(Y_{t},\,Z_{t})\}_{t \in [0,T]}  \in \mathcal{S} \times \mathcal{L}^{2} $ fulfilling that \as, \myeqref{eq1} holds and $ \int_{0}^{T} \lvert  f(u,\,Y_{u},\,Z_{u})\rvert  du < \infty $. Furthermore, we designate $ (Y,\,Z) := \{(Y_{t},\,Z_{t})\}_{t \in [0,T]}$ as a $ L^{p} $ solution provided that $ (Y,\,Z) \in \mathcal{S}^{p}\times  \mathcal{M}^{q}$ for some $ p>1,\,q>1 $. 
Now we establish the existence theorem for BSDE \myeqref{eq1}.
\begin{theorem} \label{thm:exist1} \ Assume BSDE \myeqref{eq1} satisfies that $ \xi >0 $, $ f $ is continuous in $ (y,\,z) $ for each fixed $ (s,\,\omega) $, and for any $ (s,\,\omega,\,y,\, z) \in [0,\, T]\times \Omega \times (0,\, +\infty) \times  \mathbb{R}^{1 \times d}, $
			\begin{equation}\label{eq3}
		0 \leq f(s,\,\omega,\,y,\, z) \leq a_{s}+b_{s} y+\gamma_{s} |z|+\frac{\delta}{2y}|z|^{2},  
		\end{equation}
	where $ \delta \ge 0 $ is a constant, $a,\,b$, and $ \gamma $ are nonnegative and progressively measurable processes with $  a > 0  \text{ or } b > 0$. Then BSDE \myeqref{eq1} admits a solution $ (Y,\,Z) $ such that $Y \in \mathcal{S}^{2p(1+\delta)} $ for some $ p>1 $, if
	$$\mathbb{E}\left[\exp\left\{\frac{p}{p-1}\int_{0}^{T} \gamma_{s}^{2} d s+ 2(1+\delta)p\int_{0}^{T}\left(a_{t}+b_{t}\right)d t \right\}\Big( \xi^{2p(1+\delta)}+1 \Big)\right]<+\infty. $$
	Moreover, $ Z \in \mathcal{M}^{2p} $ if $ \delta \neq 1 $.
\end{theorem}

To prove \sref{Theorem}{thm:exist1}, we recall the domination method, which is first employed by \citep{Bahlali2013} to our knowledge, and further developed in \citep{Bahlali2017, Bahlali2018} and \citep{ Bahlali201903}.

\begin{lemma}\label{lemBahlali}(\citep{Bahlali201903}, Lemma 3.1). \ 
	We say that the BSDE$ (\xi, f) $ satisfies the domination conditions if there exist two BSDEs$ (\xi^{1}, f^{1}) $ and $ (\xi^{2}, f^{2}) $ such that:
	\begin{enumerate}
		\item[$ (1) $] $ \xi^{1} \le \xi \le \xi^{2} $;
		\item[$ (2) $] BSDE$ (\xi^{1}, f^{1}) $ and BSDE$ (\xi^{2}, f^{2}) $ have solutions $ (Y^{1},\,Z^{1}) $ and $ (Y^{2},\,Z^{2})$ in $\mathcal{S} \times \mathcal{L}^{2} $, respectively, such that
		$ Y^{1}_{s} \le Y^{2}_{s},\,\forall s \in [0,\,T] $.
		For every $ (s,\,\omega) \in [0,\,T] \times \Omega$, $ y \in [Y^{1}_{s}(\omega),\, Y^{2}_{s}(\omega)]$, and $ z \in \mathbb{R}^{d}$, the following inequalities hold:
		$$ f^{1}(s,\, \omega,\,y,\,z) \le f(s,\,\omega,\,y,\,z) \le f^{2}(s,\,\omega,\,y,\,z), $$
		and
		$$
		|f(s,\, \omega,\, y,\, z)| \le \eta_{s}(\omega)+\frac{\zeta_{s}(\omega)}{2}|z|^{2},
		$$
		where $ \eta  $ and $ \zeta $ are two $ \mathcal{F}_{t}$-adapted processes satisfying that $\zeta $ is continuous and $ \forall \omega,$  $ \, \int_{0}^{T}\lvert  \eta _{t}(\omega)\rvert  dt <+\infty $.	
	\end{enumerate}
 Assume that the BSDE$ (\xi, f) $ satisfies the domination conditions and  $ f $ is continuous in $ (y,\,z) $ for each fixed $ (s,\,\omega) $. Then the BSDE$ (\xi, f) $ admits at least one solution $ (Y,\,Z) $ such that 
	\begin{equation}\label{eq4}
 Y^{1}_{s} \le Y_{s} \le Y^{2}_{s},\, \forall s \in [0,\,T],
	\end{equation}
and there exist a maximal and a minimal solutions among all solutions satisfying \myeqref{eq4}.
\end{lemma}  

	\begin{proof}[Proof of Theorem 1.]
	Assume $ \delta \ge 0.$ Consider BSDEs
	\begin{equation}\label{eq5}
	Y^{U}_{t} = \xi + \int_{t}^{T} a_{s}+ b_{s}Y^{U}_{s} + \gamma_{s}\lvert Z^{U}_{s} \rvert +\dfrac{\delta}{2Y^{U}_{s}}|Z^{U}_{s}|^{2}ds -\int_{t}^{T} Z^{U}_{s} dW_{s} , 
	\end{equation}
	and
	$$ Y^{L}_{t}=\xi-\int_{t}^{T} Z^{L}_{s} d W_{s}.$$
	Since $ \xi $ is integrable, the second BSDE has a solution $ (Y^{L},\, Z^{L}) $ denoted by $Y^{L}_{t}= \mathbb{E}\left[\xi |\mathcal{F}_{t} \right]$ according to Proposition 1.1 in \citep{ Bahlali201903}. Using \sref{Lemma}{lemBahlali} (with $ Y^{2}_{t} = Y_{t}^{U}$ and $ Y^{1}_{t} = Y_{t}^{L} $) we infer that BSDE \myeqref{eq1} has a solution $ (Y,\, Z) $ satisfying $Y_{t} \ge \mathbb{E}\left[\xi \middle | \mathcal{F}_{t} \right]$ if the BSDE \myeqref{eq5} admits a solution $ (Y^{U},\, Z^{U}) $ satisfying $Y_{t}^{U} \ge Y_{t}^{L}$. 
	
	Define an invertible function $$u(y) := \dfrac{y^{1+\delta}-1}{1+\delta} ,\quad y > 0. $$
	Then $ u(\cdot) \in \mathcal{C}^{2}((0,\, +\infty)) $. It follows from It\^o's formula and \sref{Lemma}{lemBahlali} that the BSDE \myeqref{eq5} admits a solution $ (Y_{t}^{U},\, Z_{t}^{U}) $ satisfying $Y_{t}^{U} \ge Y_{t}^{L}$ if and only if the following equation admits a solution $ (Y_{t}^{0},\, Z_{t}^{0}) $ such that $Y_{t}^{0} \ge u(Y_{t}^{L})  $:
	\begin{equation}\label{eq6}
	Y^{0}_{t} = u\left( \xi \right)+ \int_{t}^{T} a_{s} \left[\left(\delta +1 \right)Y^{0}_{s}+1\right]^{\frac{\delta}{\delta+1}} + b_{s}\left[\left(\delta +1 \right)Y^{0}_{s}+1\right] + \gamma_{s}\lvert Z^{0}_{s} \rvert ds -\int_{t}^{T} Z^{0}_{s} dW_{s}.
	\end{equation}
	So far, the proof is similar to that in
	\cite{Bahlali2018}. They used results in \cite{Bahlali2015} to prove the equation \myeqref{eq6} has a solution $ (Y^{0},\, Z^{0}) $ such that $Y_{t}^{0} \ge u(Y_{t}^{L})$, while we continue to use \sref{Lemma}{lemBahlali}. It is apparent that equation \myeqref{eq6} has a solution $ (Y^{0},\, Z^{0}) $ such that $Y_{t}^{0} \ge u(Y_{t}^{L})$
	if the following equation admits a solution $ (Y^{1},\, Z^{1}) $ such that $Y_{t}^{1} \ge u(Y_{t}^{L})  $:
	\begin{equation}\label{eq7}
	Y^{1}_{t} = \left( u\left( \xi \right)\right) ^{+} + \int_{t}^{T} a_{s} \left[\left(\delta +1 \right)Y^{1}_{s}+1\right]^{\frac{\delta}{\delta+1}} + b_{s}\left[\left(\delta +1 \right)Y^{1}_{s}+1\right] + \gamma_{s}\lvert Z^{1}_{s} \rvert ds -\int_{t}^{T} Z^{1}_{s} dW_{s}.
	\end{equation}		
	According to Fatou's lemma and Jensen's inequality, BSDE \myeqref{eq7} admits a solution $ (Y^{1},\, Z^{1}) $ such that  $Y_{t}^{1} \ge u(Y_{t}^{L})$  
	if it admits a solution $ (Y^{1},\, Z^{1}) $ such that $ Y_{t}^{1} \ge 0. $ Define another invertible function$$ H(y) := \int_{2}^{2+y} \dfrac{dr}{(1+\delta)r+1} ,\, \ y>-2. $$ 
	Then $ H(\cdot) \in \mathcal{C}^{2}((-2,\, +\infty))$. By using It\^o's formula and \sref{Lemma}{lemBahlali} again, this holds if
	\begin{align*}
	Y_{t}^{2}=H\left(\left( u\left( \xi \right) \right)^{+}\right)+\int_{0}^{T}  \left( a_{s} + b_{s}\right) d s + \int_{t}^{T}  \gamma_{s}\lvert Z_{s}^{2}\rvert    +\frac{1+\delta}{2}\lvert Z^{2}_{s} \rvert^{2}d s-\int_{t}^{T} Z_{s}^{2} d W_{s}
	\end{align*}
	admits a solution $\left(Y^{2},\, Z^{2}\right)$ such that $Y_{t}^{2} \geq \int_{0}^{t}\left( a_{s} + b_{s}\right)ds.$ This holds if 
	\begin{align}\label{eq8}
	Y_{t}^{3}=H^{-1}\left(H\left(\left( u\left( \xi \right) \right)^{+}\right)+ \int_{0}^{T} \left( a_{s} + b_{s}\right) ds\right) +\int_{t}^{T} \gamma_{s}\lvert Z_{s}^{3}\rvert   
	d s-\int_{t}^{T} Z_{s}^{3} d W_{s}
	\end{align}
	admits a solution $\left(Y^{3},\, Z^{3}\right)$ satisfying $Y^{3}_{t} \geq \mathbb{E}\left[  Y^{3}_{T} \middle | \mathcal{F}_{t} \right]$.  
	Recognizing that $ a >0 $ or $ b >0$ ensures $\int_{0}^{T} \left(a_{s} + b_{s}\right)ds > 0$, this holds if for $ p > 1, $
	\begin{equation}\label{eq10}
	Y_{t}^{4}=\left(Y^{3}_{T}\right)^{p} + \int_{t}^{T} \left( \gamma_{s}\lvert Z_{s}^{4}\rvert  - \frac{p-1}{2p}\frac{\lvert Z^{4}_{s} \rvert^{2}}{Y^{4}_{s}} \right)d s-\int_{t}^{T} Z_{s}^{4} d W_{s}	
	\end{equation}	
	admits a solution $\left(Y^{4},\, Z^{4}\right)$ fulfilling 
	$Y^{4}_{t} \geq \left(\mathbb{E}\left[  Y^{3}_{T} \middle | \mathcal{F}_{t} \right]\right)^{p}$.
	
	Since the BSDE
	$ y_{t}=\left(Y^{3}_{T}\right)^{p} -\int_{t}^{T}\frac{p-1}{2p}\frac{\lvert z_{s} \rvert^{2}}{y_{s}} d s-\int_{t}^{T} z_{s} d W_{s}$
	has a solution $ (y_{t},\, z_{t}) $ denoted by 
	$y_{t} = \left(\mathbb{E}\left[  Y^{3}_{T} \middle | \mathcal{F}_{t} \right]\right)^{p}$, we infer from Young's inequality and \sref{Lemma}{lemBahlali} that BSDE \myeqref{eq10} admits a solution $\left(Y^{4},\, Z^{4}\right)$ satisfying $Y^{4}_{t} \geq \left(\mathbb{E}\left[Y^{3}_{T}\middle | \mathcal{F}_{t}  \right]\right)^{p}  $ if the BSDE
	\begin{equation}\label{eq11}
	Y_{t}^{5} = \left(Y^{3}_{T}\right)^{p} +  \int_{t}^{T} \frac{p \gamma_{s}^{2}}{2(p-1)}Y^{5}_{s}ds -\int_{t}^{T} Z_{s}^{5} d W_{s} 
	\end{equation}
	admits a solution satisfying $Y_{t}^{5} \ge y_{t}$. It's clear from the integrability assumption that \myeqref{eq11} admits a solution $\left(Y^{5},\, Z^{5}\right)$ given by
	$$  Y_{t}^{5}=\mathbb{E}\bigg[\exp\left\{\frac{p}{2(p-1)}\int_{t}^{T} \gamma_{s}^{2} d s\right\}\left(  H^{-1}\left(H\left(\left( u\left( \xi \right) \right)^{+}\right)+\int_{0}^{T}  \left( a_{s} + b_{s}\right) d s\right)\right)^{p}\mid \mathcal{F}_{t} \bigg] \ge y_{t} .$$
	Therefore, going back to BSDE \myeqref{eq1}, we deduce that BSDE \myeqref{eq1} has a solution $ (Y,\,Z) $ such that 
$$	\mathbb{E}\left[ \xi \mid \mathcal{F}_{t} \right] \le Y_{t} \le u^{-1}\left( \left( Y_{t}^{5} \right)^{1/p} \right). 
$$
	It is evident from  Doob's $ L^{p} $ inequality and Jensen's inequality that $ Y \in \mathcal{S}^{2p(1+\delta)}$.
	
	Now we prove $  Z \in \mathcal{M}^{2p}. $ Assume that $ \delta >1 $. We use $ C $ to denote a generic positive constant throughout the proof and subsequent sections which may change from line to line. Define a function
	$$ v(y) := 
	\begin{cases}
	\dfrac{1}{\delta-1}\left(\dfrac{y^{1+\delta}-1}{\delta+1}- \dfrac{y^{2}-1}{2}\right),\quad  y > 1,\\
	\dfrac{y^{2}-1}{2}-(y-1),\quad\quad\quad\quad 0 < y \leq 1,	
	\end{cases} $$
	It is clear $ v(\cdot) \in \mathcal{C}^{2}((0,\,+\infty))$. Applying It\^o's formula to $ v(Y_{t})$ derives that 
	$$v\left(Y_{t}\right)= v(Y_{0})+\int_{0}^{t} \left(  v^{\prime}\left( Y_{r}\right)\left(- f\left(r,\, Y_{r},\, Z_{r}\right)\right)+\frac{1}{2} v^{\prime \prime}\left( Y_{r}\right)\lvert Z_{r}\rvert ^{2} \right)  d r  +\int_{0}^{t} v^{\prime}\left( Y_{r}\right) Z_{r} d W_{r}. 	$$
	By virtue of the Cauchy-Schwarz inequality and the BDG inequality, we can derive that 
	$$
	\mathbb{E}\left[\left(\int_{0}^{T} \lvert Z_{u}\rvert ^{2} d u\right)^{p}\right] \leq  C\mathbb{E}\left[1+\sup_{0 \le s \le T} (Y_{s})^{2p(1+\delta)}+\left(\int_{0}^{T}\left(a_{s} + b_{s} \right)d s \right)^{2p} +\left(\int_{0}^{T}\gamma^{2}_{t}d t \right)^{p(1+\delta)}\right]. 
	$$
	According to the integrability assumption, $  Z_{t} \in \mathcal{M}^{2p}. $
	The case of $ 0\le \delta < 1 $ can be proved similarly by taking $ v(y) =  \frac{y^{2}}{2(1-\delta)}$ for $y>0 $. This completes the proof of Theorem 1.   
\end{proof}

\begin{remark}\label{remark:2.1}
   We extend \sref{Theorem}{thm:exist1} to a more generalized BSDE \myeqref{eq1} with the generator satisfying that
	$ f $ is continuous in $ (y,\,z) $ for each fixed $ (s,\,\omega) $, and
	\begin{equation}\label{eqf2}
	0 \le f(s,\,\omega,\,y,\,z) \le a_{s}+b_{s}\phi(y)+\gamma_{s}|z|+\frac{|z|^{2}}{2\psi(y)} ,\,(s,\,\omega,\, y,\, z) \in [0,\, T] \times \Omega \times (0,\, +\infty) \times \mathbb{R}^{1 \times d},   
	\end{equation}
	where $a,\,b$, and $ \gamma $ are nonnegative and progressively measurable processes with $  a > 0  \text{ or } b > 0$. $ \phi:\left[0,\,\left.+\infty\right)\right.\mapsto \left[0,\,\left.+\infty\right)\right. $ is  convex, increasing and has continuous first derivative on $(0,\, +\infty).\,\phi(x)>0,\,\forall x>0.\, \psi:\left[0,\,\left.+\infty\right)\right. \mapsto \left[0,\,\left.+\infty\right)\right. $ satisfies $ \psi(0) = 0 $. $ \phi / \psi $ is increasing, and $ 1/\psi(x) $ is continuous on $ (0,\, +\infty) $.
	 For example, one can take the following functions $ \phi$ and $ \psi $:
	\begin{itemize}
		\item 	$   \phi (y) = \psi(y)  = \exp(y),\, y > 0; \phi (0) = \psi(0)  = 0.$
		\item   $ \phi (y) = \psi(y) = y^{2},\, y \ge 0. $
     	\item	$  \phi(y) = y,\,\psi(y) = \sqrt{y},\, y \ge 0. $
	\end{itemize}

    To present the existence results, we define the following auxiliary functions:
	$$
	\hat{u}(y):=\int_{1}^{y}\exp\left(\int_{1}^{x}\frac{dr}{\psi(r)}\right)\ dx,\  y > 0 .	$$
	$$
	\tilde{u}(y):=\int_{1}^{y}\exp\left(\int_{1}^{x}\frac{\mathds{1}_{ r>0 }}{\psi(r)}\ dr\right)\ dx,\  y > 0 .	$$
	$$ \hat{H}(y) := \int_{2}^{2+y} \dfrac{dr}{\phi(u^{-1}(r))u'(u^{-1}(r))} ,\, \ y>-2. $$
	$$K(x):=  \int_{0}^{x}\exp\left(-\int_{1}^{z} \frac{\mathds{1}_{ r>0 }}{\psi(r)}\ dr\right)\ dz,\ x \ge  0, $$ 
	$$ \hat{v}(y):=\int_{0}^{y} K(x) \exp\left(\int_{1}^{x}\frac{\mathds{1}_{ r>0 }}{\psi(r)}\ dr\right)\ dx,\  y >0.$$
	For $ p > 1,\, 0 \le s \le T $ and $ y \in (0,\, +\infty)$,
	$$
	\Phi(s,\, y; a_{\cdot},\, b_{\cdot},\, \gamma_{\cdot}):=  \exp\Big\{\frac{p}{2(p-1)}\int_{0}^{s} \gamma_{r}^{2} d s\Big\}\Big(\hat{H}^{-1}\Big(\hat{H}\Big(\left( \hat{u}\left( y \right)\right)^{+} \Big) + \int_{0}^{s}\Big(\frac{a_{r}}{\phi(1)} + b_{r}\Big)  dr\Big)\Big)^{p}.$$
	$$X_{t}:= \hat{u}^{-1}\Big(\Big\{ \mathbb{E}\big[\Phi(T,\, \xi;\, a_{\cdot},\, b_{\cdot},\, \gamma_{\cdot}) 	\mid \mathcal{F}_{t} \big] \Big\}^{1/p} \Big).
	$$
	We conclude that
	\begin{itemize}
		\item	If $ \xi > 0 $ and $\mathbb{E}\big[  \Phi^{k}(T,\, \xi;\, a_{\cdot},\, b_{\cdot},\, \gamma_{\cdot})  \big]<+\infty$ for some $ k > 1$, then BSDE \myeqref{eq1} has a solution $ (Y,\,Z) $ satisfying $ \hat{u}(Y) \in \mathcal{S}^{kp}$;
		\item	If $ \xi > 0 $, $\left(\mathds{1}_{ x>0 }/\psi(x) \right) _{x \in \mathbb{R}} $ is locally integrable on $   \mathbb{R}$, $\mathbb{E}\bigg[  \left(\int_{0}^{T}\gamma_{t}^{2} dt\right) ^{p} \bigg]<+\infty$, $ \mathbb{E}\bigg[  \phi^{p}\left(\hat{X}^{\ast}\right)  \bigg]<+\infty$, and $\mathbb{E}\bigg[  \big(\hat{X}^{\ast}\hat{u}'(\hat{X}^{\ast})\big) ^{2p} \bigg]<+\infty $, then BSDE \myeqref{eq1} has a solution $ (Y,\,Z) $ satisfying $ \hat{u}(Y) \in \mathcal{S}^{2p}$ and $ Z \in \mathcal{M}^{p}$.
	\end{itemize}	
	These results can be verified using the same method as in \sref{Theorem}{thm:exist1}. It is worth noting that we have checked that $ \hat{u}(\cdot) \in  \mathcal{C}^{1}((0,\,+\infty) ) \cap \mathcal{W}^{2}_{1,\, loc}( (0,\,+\infty)) $, which is critical for applying It\^o-Krylov's formula. The functions $ \hat{u}^{-1}(\cdot),\,\tilde{u}(\cdot),\,\tilde{u}^{-1}(\cdot),\,\hat{H}(\cdot),\,\hat{H}^{-1}(\cdot) $, and $ \hat{v}(\cdot)$ also satisfy this property.
\end{remark}

\section{Comparison Theorem}\label{sec3}

To obtain the uniqueness result, we assume a convexity condition in $ y,\, z $ so as to exploit $ \theta$-technique proposed by \citep{BriandHu2008} and proved to be convenient to deal with quadratic generators.

\begin{theorem}\label{thm:compa}
	(Comparison Theorem)
	Let $(Y^{1},\,Z^{1})$ and $ (Y^{2},\, Z^{2}) $ be solutions to BSDE($ \xi^{1},\, f^{1} $) and BSDE($ \xi^{2},\, f^{2} $), respectively, such that both $ Y^{1}$ and $Y^{2} $ belong to $ \mathcal{S}^{2p(2+\delta)} (0\le \delta \neq 1,\,p>1 )$.
	
	Assume $ d\mathbb{P} \otimes du\text{-a.s.},\, f^{1}(u,\, Y^{2}_{u},\, Z^{2}_{u}) \leq f^{2}(u,\, Y^{2}_{u},\, Z^{2}_{u})   $
	and	\as, $0 <\xi^{1} \leq \xi^{2}$. If  $ f^{1}(t,\, \cdot,\, \cdot)$ is convex for $\forall\, 0 \le t \le T$ and 
	for any $ (s,\,\omega,\,y,\, z) \in [0,\, T]\times \Omega \times (0,\, +\infty) \times  \mathbb{R}^{1 \times d}, $
$$		0 \leq f^{1}(s,\,\omega,\,y,\, z) \leq a_{s}+b_{s} y+\gamma_{s} |z|+\frac{\delta}{2y}|z|^{2},  	$$
	where $a,\,b$, and $ \gamma $ are nonnegative and progressively measurable processes satisfying
	$$\mathbb{E}\left[\exp\left\{\frac{p}{p-1}\int_{0}^{T} \gamma_{s}^{2} d s+ 2p(2+\delta)\int_{0}^{T}\left(a_{t}+b_{t}\right)d t \right\}\right]<+\infty, $$
	then \as, $\forall\, 0 \le t \le T,\, Y^{1}_{t} \leq Y_{t}^{2}$.
\end{theorem}  

\begin{proof}[Proof.]
	For any given $\theta \in (0,\, 1) $, set
	$$
	\Delta_{\theta}Y_{t}:=\frac{Y^{1}_{t}-\theta Y_{t}^{2}}{1-\theta}, \  \Delta_{\theta} Z_{t} := \frac{Z^{1}_{t}-\theta Z^{2}_{t}}{1-\theta}  \text{, and }   \Delta_{\theta}\xi := \frac{\xi^{1}-\theta \xi^{2}}{1-\theta}.
	$$ 
	Suppose $ L^{0}_{\cdot}(\Delta_{\theta}Y) $ represents the local time of $ \Delta_{\theta}Y $ at $ 0 $. We can acquire via Tanaka's formula that
	$$\begin{aligned}
	\left(\Delta_{\theta}Y_{t}\right)^{+} = &\left(\Delta_{\theta}\xi\right)^{+} + \int_{t}^{T} \mathds{1}_{ \Delta_{\theta}Y_{u}>0 }\frac{1}{1-\theta}\Big[f^{1}\left(u,\, Y^{1}_{u},\, Z^{1}_{u}\right)-\theta f^{2}\left(u,\, Y_{u}^{2},\, Z_{u}^{2}\right)\Big]du \\-& \int_{t}^{T} \mathds{1}_{ \Delta_{\theta}Y_{u}>0 }\Delta_{\theta} Z_{u}d W_{u}- \frac{1}{2}\int_{t}^{T}dL^{0}_{u}(\Delta_{\theta}Y),
	\end{aligned}$$ 
	where
	$$\begin{aligned}&\mathds{1}_{ \Delta_{\theta}Y_{u}>0 }\frac{1}{1-\theta}\Big[f^{1}\left(u,\, Y^{1}_{u},\, Z^{1}_{u}\right)-\theta f^{2}\left(u,\, Y_{u}^{2},\, Z_{u}^{2}\right)\Big]
	\\ =& \mathds{1}_{ \Delta_{\theta}Y_{u}>0 }\frac{1}{1-\theta}\left[f^{1}\left(u,\, Y^{1}_{u},\, Z^{1}_{u}\right)-\theta f^{1}\left(u,\, Y_{u}^{2},\, Z_{u}^{2}\right)+\theta f^{1}\left(u,\, Y_{u}^{2},\, Z_{u}^{2}\right)-\theta f^{2}\left(u,\, Y_{u}^{2},\, Z_{u}^{2}\right)\right] \\
	\leq & \mathds{1}_{ \Delta_{\theta}Y_{u}>0 }\Big[a_{u}+b_{u}\Delta_{\theta}Y_{u}+\gamma_{u}|\Delta_{\theta}Z_{u}|+\frac{\delta|\Delta_{\theta}Z_{u}|^{2}}{2\Delta_{\theta}Y_{u}} \Big].
	\end{aligned}$$
	
	Let $h(y):= \int_{2}^{2+y} \dfrac{dr}{(2+\delta)r+1} ,\, \ y>-2$. For ease of notation, let us denote
	$$\begin{aligned} P_{t}^{0}:=& h^{-1}\left(h\left(\frac{\left(\left(\Delta_{\theta}Y_{t}\right)^{+}\right)^{2+\delta}}{2+\delta}\right) + \int_{0}^{t}\left(a_{s} + b_{s}\right)  ds\right),  \\
	G_{t}:=&\left(\exp\left\{\frac{1}{2(p-1)}\int_{0}^{t} \gamma_{s}^{2} d s\right\}P^{0}_{t}  \right)^{p},\\
	R_{t}:=&\exp\left\{\frac{p}{2(p-1)}\int_{0}^{t} \gamma_{s}^{2} d s\right\}   p\left(P^{0}_{t} \right)^{p-1} \dfrac{\left[(2+\delta)\left(2+P^{0}_{t}\right)+1 \right] \left(\left(\Delta_{\theta}Y_{t}\right)^{+}\right)^{1+\delta}\mathds{1}_{\Delta_{\theta}Y_{t}>0}}{(2+\delta)\left(2+\frac{\left(\left(\Delta_{\theta}Y_{t}\right)^{+}\right)^{2+\delta}}{2+\delta}\right)+1}
	\Delta_{\theta} Z_{t}.
	\end{aligned} $$
	By applying It\^o's formula to $ P_{t}^{0} $ and using Lemma 3.1 in \citep{Yang2017}, we deduce that
	$$ G_{t}\leq  G_{T} -\int_{t}^{T} R_{s}d W_{s}.
	$$
	Let us define $ \sigma_{n}^{t} :=\inf\left\{ s > t: \int_{t}^{s}|R_{u}|^{2}du>n\right\}\wedge T,\,t \in [0,\, T],\,n \in \mathbb{N}. $ Therefore, $ G_{t}\leq \mathbb{E}\left[ G_{\sigma_{n}^{t}} \middle | \mathcal{F}_{t}  \right].	$
	Noting that $\left(\Delta_{\theta}\xi\right)^{+} \leq \xi^{1}  $, the dominated convergence theorem gives
	$$\left(\tfrac{\left( Y^{1}_{t}-\theta Y^{2}_{t}\right)^{+}}{1 - \theta} \right)^{p(2+\delta)}\!\! \leq  C \mathbb{E}_{t}\left[\exp\left\{\tfrac{p}{2(p-1)}\int_{0}^{T} \gamma_{s}^{2} d s+ p(2+\delta)\int_{0}^{T}\left(a_{t}+b_{t}\right)d t \right\}\left(\xi^{1}+1\right)^{p(2+\delta)}  \right] .
	$$
	Moving $ (1-\theta) $ to the right-hand side and sending $ \theta \to 1$, we complete the proof. \end{proof} 

\begin{Corollary}\label{corunique}
	There is only one solution $ (Y,\,Z)  \in \mathcal{S}^{2p(2+\delta)} \times \mathcal{M}^{2p}$ for BSDE$ (\xi,\, f) $, if $\xi >0,\, f(t,\, \cdot,\, \cdot)$ is convex, $\forall 0 \le t \le T$ and  for any $ (s,\,\omega,\,y,\, z) \in [0,\, T]\times \Omega \times (0,\, +\infty) \times  \mathbb{R}^{1 \times d}, $
	$$	0 \leq f(s,\,\omega,\,y,\, z) \leq a_{s}+b_{s} y+\gamma_{s} |z|+\frac{\delta}{2y}|z|^{2},  	$$
	where $ 0\le \delta \neq 1 $ is a constant, $a,\,b$, and $ \gamma $ are nonnegative and progressively measurable processes satisfying $ a > 0 $ or $ b > 0 $, and for some $ p >1 $,
	$$\mathbb{E}\left[\exp\left\{\frac{p}{p-1}\int_{0}^{T} \gamma_{s}^{2} d s+ 2p(2+\delta)\int_{0}^{T}\left(a_{t}+b_{t}\right)d t \right\}\Big( \xi^{2p(2+\delta)}+1 \Big)\right]<+\infty. $$
	More precisely,
	$$\begin{aligned}
	&\mathbb{E}\left[\sup_{0\leq t \leq T}Y_{t}^{2p(2+\delta)}  + \left(\int_{0}^{T} \lvert Z_{u}\rvert ^{2} d u\right)^{p} \right]  \\ \le  &C\mathbb{E}\left[\exp\left\{\frac{p}{p-1}\int_{0}^{T} \gamma_{s}^{2} d s+ 2p(2+\delta)\int_{0}^{T}\left(a_{t}+b_{t}\right)d t \right\}\Big( \xi^{2p(2+\delta)}+1 \Big)\right],   
	\end{aligned}$$
	where $ C $ is a constant depending on $ p$ and $\delta. $
\end{Corollary}

\begin{remark}
	For more generalized BSDE in \sref{Remark}{remark:2.1}, a similar comparison theorem holds. In fact, let $(Y^{1},\,Z^{1})$ and $ (Y^{2},\, Z^{2}) $ be solutions to BSDE($ \xi^{1},\, f^{1} $) and BSDE($ \xi^{2},\, f^{2} $), respectively, fulfilling for given $ p > 1 ,$
	$$ \forall \lambda > 0,\, \mathbb{E}\left[\sup_{0\leq t \leq T}   \Phi(t,\, \lambda Y_{t}^{i};\, a_{\cdot},\, b_{\cdot},\, \gamma_{\cdot})  \right] < +\infty,\, i=1,\,2. $$	
	Assume $ d\mathbb{P} \otimes du\text{-a.s.},\, f^{1}(u,\, Y^{2}_{u},\, Z^{2}_{u}) \leq f^{2}(u,\, Y^{2}_{u},\, Z^{2}_{u})$, and	\as, $0 < \xi^{1} \leq \xi^{2}$. If  $ f^{1}$ satisfies \myeqref{eqf2}, $\left(\mathds{1}_{ x>0 }/\psi(x) \right) _{x \in \mathbb{R}} $ is locally integrable on $   \mathbb{R}$, and $ f^{1}$ is convex in $(y,z)$ for all $ t \in [0,\,T]$, then \as, $\forall\, 0 \le t \le T,\, Y^{1}_{t} \leq Y_{t}^{2}$.
	
	Moreover, if BSDE($ \xi,\, f $) satisfies \myeqref{eqf2}, $\left(\mathds{1}_{ x>0 }/\psi(x) \right) _{x \in \mathbb{R}} $ is locally integrable on $   \mathbb{R}$, $ f$ is convex in $(y,z)$ for all $ t \in [0,\,T]$, for some $ p > 1$, 
	$$\mathbb{E}\left[ \sup_{0 \le t \le T} \phi^{p/2}\left(X_{t} \right)  \right] <+\infty, $$	 
	and for any $ \lambda \ge 1 $, $$\mathbb{E}\left[ \sup_{0 \le t \le T} \exp\left\{\frac{p}{2(p-1)}\int_{0}^{t} \gamma_{s}^{2} d s\right\}  \Big(\hat{H}^{-1}\Big(\hat{H}\big(\lambda X_{t} \hat{u}'\left(\lambda X_{t} \right) \big) +   \int_{0}^{t}\Big(\frac{a_{s}}{\phi(1)} \!+\!b_{s}\Big)  ds\Big)\Big)^{p}  \right] <\infty,$$
	then BSDE($ \xi,\, f $) has a unique solution $ (Y,\,Z) $ in $ \mathcal{S}^{p} \times \mathcal{M}^{p/2} $. 	
\end{remark}

\begin{Corollary}\label{cor:bmo}
	(BMO property) 	
	Assume that BSDE($ \xi,\, f $) satisfies that $\xi >0 $ is bounded, \myeqref{eq3} holds with $0\le \delta \neq 1,\,\int_{0}^{T} a_{s} d s,\, \int_{0}^{T} b_{s} d s$ and $ \int_{0}^{T}\gamma_{s}^{2}d s$ are bounded. If $ f(t,\, \cdot,\, \cdot)$ is convex for $\forall\, 0 \le t \le T$, then there is only one solution $(Y,\,Z) $ for BSDE($ \xi,\, f $) fulfilling that $ Y $ is bounded and $ Z  \circ W \in BMO$. 
\end{Corollary}

\begin{proof}[Proof.]
	\sref{Corollary}{corunique} provides clear evidence of the existence, uniqueness of solution $(Y,\,Z) $ together with the boundedness of $ Y_{t} $. 
	Applying It\^o's formula to $ v(Y_{t}) $, where $ v(\cdot) $ defined in the proof of \sref{Theorem}{thm:exist1} and then taking conditional expectation, the argument $ Z \circ W \in BMO $ can be easily proved by localization technique and dominated convergence theorem.
\end{proof}
\begin{remark}
	It is clear that for more generalized BSDE in \sref{Remark}{remark:2.1}, \sref{Corollary}{cor:bmo} still holds.
\end{remark}

\section{Stability Result and Application to PDEs}\label{sec4}
\subsection{Stability Result}

The following result establishes a stability result for the solution to BSDE$ (\xi,\, f)$. To be more detailed, suppose that $ f $ is a generator satisfying \myeqref{eq3}  with $ a_{\cdot},\, b_{\cdot},\, \gamma_{\cdot}$ and $ f(t,\, \cdot,\, \cdot)$ is convex for $\forall\, 0 \le t \le T$. For each $ n \ge 1,\, f_{n} $ is a generator fulfilling \myeqref{eq3} with $(a^{n}_{\cdot},\, b^{n}_{\cdot},\, \gamma^{n}_{\cdot})$ and $ f_{n}(t,\, \cdot,\, \cdot)$ is convex for $\forall\, 0 \le t \le T$. $ \delta \ge 0 $ is the constant in \myeqref{eq3}. For simplicity of representation, we denote
$$A_{t}:= \exp\left\{\frac{p}{p-1}\int_{0}^{t} \gamma_{s}^{2} d s+ p(2+\delta)\int_{0}^{t}\left(a_{s}+b_{s}\right)d s \right\},$$ and
$$A^{n}_{t}:= \exp\left\{\frac{p}{p-1}\int_{0}^{t} \left(\gamma^{n}_{s}\right)^{2} d s+ p(2+\delta)\int_{0}^{t}\left(a^{n}_{s}+b^{n}_{s}\right)d s \right\}.$$
Assume that $ \xi $ and $ (\xi_{n})_{n \ge 1} $ are strictly positive
terminal conditions such that  for some $ p > 1, $
\begin{equation}\label{eq13}
\mathbb{E}\left[A^{4}_{T}\Big( \xi^{2+\delta}+1 \Big)^{4p}\right]+ \sup_{n \ge 1}\mathbb{E}\left[(A^{n}_{T})^{4}\Big( \left( \xi_{n}\right)^{2+\delta}+1 \Big)^{4p}\right]<+\infty.
\end{equation}
Owing to \sref{Theorem}{thm:exist1}, let $ (Y,\,Z) \in \mathcal{S}^{4p(2+\delta)} \times \mathcal{M}^{4p}$ solve BSDE($ \xi,\, f $) and $ (Y^{n},\, Z^{n}) \in \mathcal{S}^{4p(2+\delta)} \times \mathcal{M}^{4p}$ solve BSDE($ \xi_{n},\, f_{n} $) for each $n \ge 1$. 
\begin{theorem}	(Stability)\label{stability}
	If $ \xi_{n} \to \xi $  in $ L^{2p(\delta + 2)}$ and $ \int_{0}^{T} \lvert f_{n}(r,\, Y_{r},\, Z_{r})-f(r,\, Y_{r},\, Z_{r}) \rvert \ d r \to 0$ in $ L^{2p(\delta + 2)}$ as $ n \to \infty, $ then $Y^{n}$ converges to $ Y $ in  $ \mathcal{S}^{4p} $. When $ 0 \le \delta <1 $, $ Z^{n} $ converges to $ Z $ in $ \mathcal{M}^{2p} $.
	
\end{theorem}
\begin{proof}[Proof.]
	First, we need to verify that
	$\left( \sup_{0 \le t \le T}\lvert  Y^{n}_{t} - Y_{t} \rvert^{4p}\right)_{n \ge 1}$ is uniformly integrable.
	According to \sref{Corollary}{corunique},
$$\sup_{n \ge 1}\mathbb{E}\left[\sup_{t \in [0,\,T] }\lvert  Y^{n}_{t} - Y_{t} \rvert^{4p(2+\delta)} \right] \le  C \sup_{n \ge 1}\mathbb{E}\left[\sup_{t \in [0,\,T] }\lvert  Y^{n}_{t}  \rvert^{4p(2+\delta)} \right] + C\mathbb{E}\left[\sup_{t \in [0,\,T] }\lvert  Y_{t}  \rvert^{4p(2+\delta)} \right]< \infty, $$
thus $\left( \sup_{t \in [0,\,T]}\lvert  Y^{n}_{t} - Y_{t} \rvert^{4p}\right)_{n \ge 1}$ is uniformly integrable by de la Vall\'ee Poussin's criterion. 

	Our task now is to prove
	$\text{ as } n \to \infty,\,\sup_{0 \le t \le T}\lvert  Y^{n}_{t} - Y_{t} \rvert  \to 0 \text{ in } \mathbb{P} . $
	For any given $ 0\le s \le T,\,\theta \in (0,\, 1)$, and $n \ge 1,$ we define the following notation: 
	$$\begin{aligned}
	& \delta_{\theta}\xi := \frac{\xi-\theta \xi_{n}}{1-\theta}, \, \delta_{\theta}Y_{s}:=\frac{Y_{s}-\theta Y^{n}_{s}}{1-\theta},\, \  \delta_{\theta} Z_{s} := \frac{Z_{s}-\theta Z^{n}_{s}}{1-\theta},   \\ &
	\delta f_{s} := f(s,\, Y_{s},\, Z_{s})- f_{n}(s,\, Y_{s},\, Z_{s}),\\ &
	R_{s}:=\frac{\left((\delta_{\theta}Y_{s})^{+}\right)^{2+\delta}}{2+\delta}.
	\end{aligned}$$ 
	
	Applying Tanaka-Meyer-It\^o formula and Lemma 3.1 in \citep{Yang2017}, we have due to assumptions on $ f_{n} $ and the Young's inequality that
	$$\begin{aligned}
	R^{p}_{t}A^{n}_{t}  \le  &R^{p}_{T}A^{n}_{T}+  \int_{t}^{T}  pA^{n}_{s} R_{s}^{p-1}
	\left((\delta_{\theta}Y_{s})^{+}\right)^{1+\delta}\frac{\lvert \delta f(s)\rvert}{1-\theta} + pA^{n}_{s} R_{s}^{p-1}a^{n}_{s}\\-&\frac{p(p-1)}{4}A^{n}_{s}\mathds{1}_{ R_{s}>0 } R_{s}^{p-2}\left((\delta_{\theta}Y_{s})^{+}\right)^{2+2\delta} \lvert \delta_{\theta}Z_{s} \vert^{2}d s \\+ &\int_{t}^{T}pA^{n}_{s} R_{s}^{p-1}\left((\delta_{\theta}Y_{s})^{+}\right)^{1+\delta}\delta_{\theta}Z_{s}dW_{s}, \, 0 \le t \le T. 
	\end{aligned}$$
	Denote
	$$ 
	N_{t} := \int_{0}^{t}pA^{n}_{s}  R_{s}^{p-1}\left((\delta_{\theta}Y_{s})^{+}\right)^{1+\delta}\delta_{\theta}Z_{s}dW_{s}. $$ 
	Due to \myeqref{eq13} together with integrability properties of $ Y,\, Z$ and $ Z^{n} $, we can derive that $ N $ is a martingale. Thus we have
	$$\begin{aligned}
	&\frac{p(p-1)}{4}\mathbb{E}\left[ \int_{t}^{T}A^{n}_{s}\mathds{1}_{R_{s}>0} R_{s}^{p-2}\left((\delta_{\theta}Y_{s})^{+}\right)^{2+2\delta} \lvert \delta_{\theta}Z_{s} \vert^{2}d s\right] \\ \leq& \mathbb{E}\left[R^{p}_{T}A^{n}_{T} \right] + \mathbb{E}\left[ \int_{t}^{T}  pA^{n}_{s} R_{s}^{p-1}
	\left((\delta_{\theta}Y_{s})^{+}\right)^{1+\delta}\frac{\lvert \delta f(s)\rvert}{1-\theta} \!+\! pA^{n}_{s} R_{s}^{p-1}a^{n}_{s}d s \right],
	\end{aligned}$$ 
	and it follows from the BDG inequality that
	\begin{align}\label{eq14}
	\mathbb{E}\left[\sup_{0 \le t \le T}R^{p}_{t}A_{t}^{n} \right] \leq &  \mathbb{E}\left[  R^{p}_{T}A^{n}_{T} \right] + \mathbb{E}\left[ \int_{0}^{T}  pA^{n}_{s} R_{s}^{p-1}
	\left((\delta_{\theta}Y_{s})^{+}\right)^{1+\delta}\frac{\lvert \delta f_{s}\rvert}{1-\theta}ds\right]\notag\\ +&\mathbb{E}\left[ \int_{0}^{T} pA^{n}_{s} R_{s}^{p-1}a^{n}_{s}ds \right]+c(1)\mathbb{E}\left[ \langle N\rangle_{T}^{1/2}\right].  
	\end{align}
	Utilizing Young's inequality
	, a routine computation gives rise to
	\begin{align}\label{eq15}
	c(1)\mathbb{E}\left[ \langle N\rangle_{T}^{1/2}\right]\leq & \frac{1}{2} \mathbb{E}\left[\sup_{0 \le t \le T}R^{p}_{t}A^{n}_{t} \right]+\frac{2pc^{2}(1)}{p-1}\Bigg\{\mathbb{E}\Big[R^{p}_{T}A^{n}_{T} \Big] \notag \\+& \mathbb{E}\left[ \int_{t}^{T}  \left(pA^{n}_{s} R_{s}^{p-1}
	\left((\delta_{\theta}Y_{s})^{+}\right)^{1+\delta}\frac{\lvert \delta f_{s}\rvert}{1-\theta} + pA^{n}_{s} R_{s}^{p-1}a^{n}_{s}\right)d s \right]\Bigg\}.
	\end{align}
	By substituting \myeqref{eq15} to  \myeqref{eq14} and denoting $ C_{1} := 2(1+\frac{2pc^{2}(1)}{p-1})$, we have
	\begin{equation}\label{eq16}
	\mathbb{E}\left[\sup_{0 \le t \le T}R^{p}_{t}A^{n}_{t} \right] \leq  C_{1}\mathbb{E}\left[  R^{p}_{T}A^{n}_{T} +\! \int_{0}^{T} pA^{n}_{s} R_{s}^{p-1}\left(  
	\left((\delta_{\theta}Y_{s})^{+}\right)^{1+\delta}\frac{\lvert \delta f_{s}\rvert}{1-\theta}+a^{n}_{s}\right)ds \right].
	\end{equation}
	Analogously, employing the Young's inequality successively to the last two terms in \myeqref{eq16} yields that
	\begin{equation*}
	\mathbb{E}\left[\sup_{0 \le t \le T}R^{p}_{t}A^{n}_{t} \right] \leq  C\mathbb{E}\left[  R^{p}_{T}A^{n}_{T}  +  \left( \int_{0}^{T}  (A^{n}_{s})^{\frac{1}{p(2+\delta)}} \frac{\lvert \delta f_{s}\rvert}{1-\theta}ds \right)^{p(2+\delta)}\!\! \!\! + \left(\int_{0}^{T} \left(A^{n}_{s}\right)^{\frac{1}{p}}a^{n}_{s}ds\right)^{p} \right] .
	\end{equation*}
	Since $\sup_{n \ge 1}\mathbb{E}\left[(A_{T}^{n})^{4}\right] <+\infty $, we deduce from Young's inequality that 
	$$
	\mathbb{E}\left[\sup_{0 \le t \le T}\left( Y_{t}-\theta Y^{n}_{t} \right)^{+} \right] \leq  (1-\theta)C\left\{1+\mathbb{E}\left[ \left(\frac{(\xi-\theta \xi_{n})^{+}}{1-\theta} \right)^{2p(2+\delta)}  \!\!\!+ \left( \int_{0}^{T}  \frac{\lvert \delta f_{s}\rvert}{1-\theta}ds \right)^{2p(2+\delta)}\right] \right\}. 
	$$
	
	Interchanging $Y^{n} $ and $ Y$ and similar deductions then lead to
	$$
	\mathbb{E}\left[\sup_{0 \le t \le T}\left( Y^{n}_{t}-\theta Y_{t} \right)^{+} \right] \leq (1-\theta)C\left\{1+ \mathbb{E}\left[ \left(\frac{(\xi_{n}-\theta \xi)^{+}}{1-\theta} \right)^{2p(2+\delta)}  \!\!\! + \left( \int_{0}^{T}  \frac{\lvert \delta f_{s}\rvert}{1-\theta}ds \right)^{2p(2+\delta)}\right]\right\}.
	$$
	Finally, it can be easily checked that 
	$$\begin{aligned}
	&	\mathbb{E}\left[\sup_{0 \le t \le T}\lvert Y^{n}_{t}- Y_{t} \rvert \right] \\ \leq & (1-\theta)C\left\{1+ \mathbb{E}\left[\left( \int_{0}^{T}  \frac{\lvert \delta f_{s}\rvert}{1-\theta}ds \right)^{2p(2+\delta)}+\left(  \frac{\lvert\xi_{n}- \xi\rvert}{1-\theta}\right) ^{2p(2+\delta)}  + \xi^{2p(2+\delta)}  + \xi_{n}^{2p(2+\delta)}  \right] \right\}. 
	\end{aligned}$$
	Recall that $\lvert\xi-\xi_{n}\rvert $ and $\int_{0}^{T} \lvert \delta f_{s}\rvert ds $ converge to 0 in $ L^{2p(2+\delta)} $ space as $ n \to \infty, $ consequently,
	$$ \lim_{n \to \infty} \mathbb{E}\left[\sup_{0 \le t \le T}\lvert Y^{n}_{t}- Y_{t} \rvert \right] \leq  (1-\theta)C\left\{1+\mathbb{E}\left[ \xi^{2p(2+\delta)} \right]\right\}. $$
	By sending $ \theta \to 1, $ we have $ \sup_{0 \le t \le T}\lvert Y^{n}_{t}- Y_{t} \rvert $ converges to 0 in probability.  Thus we arrive at the conclusion that $ \lim\limits_{n \to \infty} \mathbb{E}\left[\sup_{0 \le t \le T}\lvert Y^{n}_{t}- Y_{t} \rvert^{4p} \right] =0 $. 
	
	Now we prove that $ \lim\limits_{n \to \infty} \mathbb{E}\left[\left( \int_{0}^{T}\lvert Z^{n}_{t}- Z_{t} \rvert^{2} \ dt\right) ^{p} \right] =0 $ when $ 0 \le \delta <1 $.
	Set $ \ell(y) := \dfrac{y^{2}}{2(1-\delta)},\ y\in \mathbb{R}$. It is clear that $ \ell(\cdot) \in \mathcal{C}^{2}(\mathbb{R}) $. Applying It\^o's formula to $ \ell((\delta_{\theta}Y_{t})^{+}) $ yields that
	$$\begin{aligned}
	\ell((\delta_{\theta}Y_{t})^{+}) \le& \ell((\delta_{\theta}\xi)^{+}) +\int_{t}^{T}\left(\ell'((\delta_{\theta}Y_{s})^{+}) \mathds{1}_{ \delta_{\theta}Y_{s}>0 }\left(\frac{\lvert\delta f_{s}\rvert}{1-\theta} + a^{n}_{s} + b^{n}_{s}\delta_{\theta}Y_{s}+\gamma^{n}_{s}\lvert \delta_{\theta}Z_{s}\rvert\right)\right.\\-&\left.\frac{1}{2}\mathds{1}_{ \delta_{\theta}Y_{s}>0 }\lvert \delta_{\theta}Z_{s}\rvert^{2} \right)  ds  -\int_{t}^{T}\ell'((\delta_{\theta}Y_{s})^{+})\delta_{\theta}Z_{s}dW_{s}.
	\end{aligned}
	$$
	By means of the same way as \sref{Theorem}{thm:exist1}, we derive that
	$$ \begin{aligned}
	&	\mathbb{E}\left[\left( \int_{0}^{T}\mathds{1}_{ \delta_{\theta}Y_{t}>0 }\lvert  Z_{t}-\theta Z^{n}_{t}\rvert^{2} \ dt\right) ^{p} \right]\\ \le & (1-\theta)^{2p}C\left\{1+ \mathbb{E}\left[\sup\limits_{0 \le t \le T}\left( (\delta_{\theta}Y_{t})^{+}\right)^{4p} \right] +
	\mathbb{E}\left[ \left( \int_{0}^{T}  \frac{\lvert \delta f_{s}\rvert}{1-\theta}ds \right)^{2p}\right]\right.\\+&\left.
	\mathbb{E}\left[\left( \int_{0}^{T} a^{n}_{s} \ ds \right)^{2p}+\left( \int_{0}^{T} b^{n}_{s} \ ds \right)^{2p} +\left( \int_{0}^{T}\left(\gamma^{n}_{s}\right)^{2}\ ds \right)^{2p}\right] 
	\right\}.
	\end{aligned}$$
	Owing to the convergence of $ Y_{n} $, $ f_{n} $ and \myeqref{eq13}, it follows that
	$$\lim\limits_{n \to \infty} \mathbb{E}\left[\left( \int_{0}^{T}\mathds{1}_{ Y_{t}>\theta Y^{n}_{t} }\lvert  Z_{t}-\theta Z^{n}_{t}\rvert^{2} \ dt\right) ^{p} \right] \le C(1-\theta)^{2p} .$$
	Similarly we have $$\lim\limits_{n \to \infty} \mathbb{E}\left[\left( \int_{0}^{T}\mathds{1}_{ Y^{n}_{t}>\theta Y_{t} }\lvert  Z_{t}^{n}-\theta Z_{t}\rvert^{2} \ dt\right) ^{p} \right] \le C(1-\theta)^{2p}.$$
	Noting the facts that
	$$ \begin{aligned}
	&\mathbb{E}\left[\left( \int_{0}^{T}\lvert  Z_{s}-\theta Z^{n}_{s}\rvert^{2} \ ds\right) ^{p} \right] \\ \le &  C\left\{\mathbb{E}\left[\left( \int_{0}^{T}\mathds{1}_{ Y_{s}>\theta Y^{n}_{s} }\lvert  Z_{s}-\theta Z^{n}_{s}\rvert^{2} \ ds\right) ^{p} \right. +
	\left( \int_{0}^{T}\mathds{1}_{ Y^{n}_{t}>\theta Y_{t} }\lvert  Z_{t}^{n}-\theta Z_{t}\rvert^{2} \ dt\right) ^{p} \right. \\+& \left. (1-\theta)^{2p}
	\left( \int_{0}^{T}\lvert  Z_{t}^{n} \rvert^{2} \ dt\right) ^{p}  +
	(1-\theta)^{2p}
	\left. \left( \int_{0}^{T}\lvert  Z_{t} \rvert^{2} \ dt\right) ^{p} \right]\  \right\},
	\end{aligned}
	$$ and
	$$ \begin{aligned}
	&\sup\limits_{n\ge1}\mathbb{E}\left[\left( \int_{0}^{T}\lvert Z^{n}_{t}\rvert^{2} \ dt\right) ^{p} \right]  \le C \sup\limits_{n\ge1} \Bigg\{ 1 + \mathbb{E}\Bigg[\sup\limits_{0 \le t \le T}\left( Y_{t}^{n}\right)^{4p}  \\ + &\left( \int_{0}^{T}a^{n}_{t} \ dt\right) ^{2p} + \left( \int_{0}^{T}b^{n}_{t} \ dt\right) ^{2p} +\left( \int_{0}^{T}\left( \gamma_{t}^{n}\right) ^{2} \ dt\right) ^{2p} \Bigg]\  \Bigg\},	\end{aligned}$$ 
	we obtain that	
	$$ 
	\lim\limits_{n\to \infty}\mathbb{E}\left[\left( \int_{0}^{T}\lvert  Z_{s}- Z^{n}_{s}\rvert^{2} \ ds\right) ^{p} \right] \le   C(1-\theta)^{2p}.
	$$
	It is worth noting that the generic constant $ C $ throughout the proof is independent of $ \theta. $ Consequently, we infer that $ Z^{n} $ converges to $ Z $ in $ \mathcal{M}^{2p} $ by sending $ \theta \to 1  $. Now the proof is completed.
\end{proof}
\begin{remark}
	It's a pity that we cannot prove $ Z^{n} $ converges to $ Z $ in $ \mathcal{M}^{2p} $ when $ \delta>1$ due to we don't find an auxiliary function $ v(\cdot) $ satisfying that $  v'(y)\frac{\delta}{2y}-v''(y)/2 =-c, y>0 \text{ for some } c >0$ and $ v'(y)\ge 0, y \ge 0. $ Solving this differential equation, the difficulty is that $ v'(y)\ge 0 $ if and only if $ y\ge \left(  2ce^{-\delta \bar{c}} \right) ^{1/(\delta-1)} >0,$ where $ \bar{c} $ is an arbitrary constant.
\end{remark}

\subsection{Application to Singular Quadratic PDEs}

In this subsection, we investigate singular quadratic PDEs associated with the BSDEs we study. To be specific, we prove the related nonlinear Feynman-Kac formula as well as the uniqueness of viscosity solutions. Consider the semilinear PDE for $ z \in \mathbb{R}^{n},\, 0 \le s \le T, $
\begin{equation}\label{eq17}
-\mathscr{L} u(s,\, z)-\partial_{s} u(s,\, z)-f\left(s,\, z,\, u(s,\, z),\, (\left(  \partial_{z}u\right) ^{\top}\sigma )(s,\, z)\right)=0, \quad u(T,\, \cdot)=\Psi(\cdot),
\end{equation}
in which $\mathscr{L}$ refers to the infinitesimal generator of  $X^{t',\, x'}$, the solution to the following SDE for each fixed $ (t',\, x') \in [0,\, T] \times \mathbb{R}^{n}$
\begin{equation}\label{eq18}
X_{t}=x'+\int_{t'}^{t} B\left(r,\, X_{r}\right) d r+\int_{t'}^{t} \sigma\left(r,\, X_{r}\right) d W_{r},\, \ t' \leq t \leq T, \quad X_{t}=x', \ 0 \le t \leq t'. 
\end{equation}
In addition, let $(Y^{t',\,x'},\, Z^{t',\,x'})$  solve the BSDE
\begin{equation}\label{eq20}
Y_{s} = \Psi(X_{T}^{t',\, x'}) + \int_{s}^{T}f(r,\, X_{r}^{t',\, x'},\, Y_{r},\,Z_{r}) dr - \int_{s}^{T}Z_{r} dW_{r}, \ 0\leq s\leq T.
\end{equation}
Our first goal is to verify \myeqref{eq17} has a viscosity solution defined by
\begin{equation}\label{eq19}
u(s,\, x) := Y_{s}^{s,\,x},\,  s\in 0,\le s \le T,\,x \in  \mathbb{R}^{n}.
\end{equation}

First, we retrospect the definition of viscosity solutions to \myeqref{eq17}. 
\begin{definition}
	A function $ u $ that is continuous on $ [0,\, T] \times \mathbb{R}^{n} $ is defined as a viscosity subsolution (resp. supersolution) to \myeqref{eq17} if for each $ (s,\, x) \in \left[ 0,\left. \right. T\right) \times \mathbb{R}^{n} $ and $\varpi \in \mathcal{C}^{1,2}([0,\, T] \times \mathbb{R}^{n}) $ satisfying $ (s,\, x) $ is a local minimum (resp. maximum) of $ \varpi - u $, 
	$$ -\partial_{s} \varpi(s,\, x)-\mathscr{L} \varpi(s,\, x)-f\left(s,\, x,\, \varpi(s,\, x),\, (\left(  \partial_{x} \varpi\right) ^{\top}\sigma)(s,\, x)\right)\le 0 \quad \left(\text{  resp. } \ge 0\right) $$ 
	as well as
	$$ \forall x \in \mathbb{R}^{n}, \,u(T,\, x) 
	\leq  \Psi(x) \quad\left(\text{ resp. }\ge \right).$$
	Furthermore, $ u $ is termed as a viscosity solution if it is a viscosity supersolution as well as a viscosity subsolution. 
\end{definition}
Now we give our assumptions.

\noindent\textbf{(A.1)} $ B: [0,\, T] \times \mathbb{R}^{n} \mapsto \mathbb{R}^{n} $ as well as $\sigma: [0,\, T] \times \mathbb{R}^{n} \mapsto \mathbb{R}^{n \times d} $ are continuous functions fulfilling for all $ (s,\, z^{1},\, z^{2}) \in [0,\, T] \times \mathbb{R}^{n} \times \mathbb{R}^{n},$ 
$$\vert B(s,\, 0)\rvert + \vert \sigma(s,\, z^{1})\rvert \le J $$ and
$$\vert B(s,\, z^{1}) - B(s,\, z^{2})\rvert + \vert \sigma(s,\, z^{1}) - \sigma(s,\, z^{2})\rvert \le J\vert z^{1} - z^{2} \rvert  $$
are met for some constant $ J >0 $.

\noindent\textbf{(A.2)} $ f: [0,\, T] \times \mathbb{R}^{n} \times \mathbb{R}\times \mathbb{R}^{1 \times d} \mapsto \mathbb{R} $ and $ \Psi: \mathbb{R}^{n} \mapsto (0,\, +\infty) $ are continuous functions fulfilling for all $ (s,\, x,\, y,\, z) \in [0,\, T] \times \mathbb{R}^{n} \times (0,\, +\infty) \times \mathbb{R}^{1 \times d}$, 
$$\begin{aligned}
&(y,\, z) \mapsto  f(s,\, y,\, z)  \text{ is convex }, \\
&\Psi(x) \le k\left(1+\lvert x \rvert^{q}\right),  \\ 
\text{and \qquad}&	0 \le  f(s,\, x,\, y,\, z)  \le k\left(1+\lvert x \rvert^{q}+ y + \lvert z\rvert  \right) +\frac{\delta|z|^{2}}{2y} 
\end{aligned}$$
are met for some constants  $ 0 \le \delta \neq 1,\, q \in \left[\left.1,\, 2\right)\right.$, and $ k \ge 0.$  

Given (A.1) and constants $p \ge 1,\, t' \in [0,\, T],\, x' \in  \mathbb{R}^{n} $, there exists a unique solution $ X^{ t',\, x'} \in \mathcal{S}^{p}( \mathbb{R}^{n}) $ to SDE \myeqref{eq18}.
Moreover, since $ \sigma $ is bounded, we can obtain similarly to \citep{BriandHu2008} that for any $ q \in \left[\left.1,\, 2\right)\right., $

$$\forall \lambda >0,\, \mathbb{E}\left[ \sup_{0 \le t \le T} \exp \bigg\{ \lambda \lvert X_{t}^{t',\, x'}\rvert^{q} \bigg\} \right] \le C \exp\bigg\{\lambda C \lvert x' \rvert^{q}\bigg\}.$$
Consequently, we infer from (A.2) together with \sref{Corollary}{corunique} that BSDE \myeqref{eq20} admits only one solution $ (Y,\, Z) \in \mathcal{S}^{p} \times \mathcal{M}^{p},\, \forall p \ge 1.$ 
\begin{theorem}
	Assuming that conditions (A.1) and (A.2) are met, the function $ u $ defined in \myeqref{eq19} serves as a viscosity solution for \myeqref{eq17} and fulfills the following inequality
\begin{equation}\label{eq21}
	0 < u(t', x') \leq C\exp\bigg\{C \lvert x' \rvert^{q}\bigg\},\quad \forall  t' \in [0 , \, T], x' \in \mathbb{R}^{n}.
	\end{equation}
	
\end{theorem}

\begin{proof}[Proof.]
	Let us first prove the continuity of the map $ (t,\, x) \mapsto u(t,\, x) $. Assume that the point sequence $ (t_{n},\, x_{n})_{n \ge 1} $ converges to $ (t,\, x) $ as $ n \to +\infty. $ 
	We claim that $$ \forall p > 0, \lim\limits_{n \to \infty} \mathbb{E}\left[\lvert \Psi(X_{T}^{t_{n},\, x_{n}})-\Psi(X_{T}^{t,\, x})\rvert^{p}\right] = 0.$$
	This is obvious from 
	$\lim\limits_{n \to \infty} \mathbb{E}\left[\lvert X_{T}^{t_{n},\, x_{n}}-X_{T}^{t,\, x}\rvert^{p}\right] = 0,\, \forall p > 0 $ and the Vitali convergence theorem. Now we prove that for any $ p > 0,  $
\begin{equation}\label{eq22}
	\lim_{n \to \infty} \mathbb{E}\left[\left(\int_{0}^{T}\big \lvert  f(s,\, X_{s}^{t_{n},\, x_{n}},\, Y_{s}^{t,\, x},\, Z_{s}^{t,\, x})-f(s,\, X_{s}^{t,\, x},\, Y_{s}^{t,\, x},\, Z_{s}^{t,\, x})\big\rvert ds\right)^{p}\right] = 0
	\end{equation}
	by the same token.
	Due to the continuity of $ f $, it follows that 
	$$ \int_{0}^{T}\big\lvert  f(s,\, X_{s}^{t_{n},\, x_{n}},\, Y_{s}^{t,\, x},\, Z_{s}^{t,\, x})-f(s,\, X_{s}^{t,\, x},\, Y_{s}^{t,\, x},\, Z_{s}^{t,\, x})\big\rvert ds \to 0 \text{ in } \mathbb{P}. $$
	Note the fact that 
	$$\begin{aligned}
	&\sup_{n \ge 1} \mathbb{E}\left[\left(\int_{0}^{T}\big\lvert f(s,\, X_{s}^{t_{n},\, x_{n}},\, Y_{s}^{t,\, x},\, Z_{s}^{t,\, x})-f(s,\, X_{s}^{t,\, x},\, Y_{s}^{t,\, x},\, Z_{s}^{t,\, x})\big\rvert ds\right)^{p}\right] \\ \le & C\left\{ 1 + \lvert x \rvert^{pq}+ \sup_{n \ge 1}\lvert x_{n}\rvert^{pq} +  
	\mathbb{E}\left[\left(\int_{0}^{T}\frac{\lvert Z_{s}^{t,\, x}\rvert^{2}}{Y_{s}^{t,\, x}}ds\right)^{p}\right] + \ \mathbb{E}\left[ \sup_{0 \le s \le T} \left(Y_{s}^{t,\, x} \right)^{p}\right]\right\}.
	\end{aligned}$$
	In order to estimate $  \mathbb{E}\left[\left( \int_{0}^{T}\frac{\lvert Z_{s}^{t,\, x}\rvert^{2}}{Y_{s}^{t,\, x}}ds \right)^{p}\right]$, we define 
	$$
	I(y) := 
	\begin{cases}
	\frac{y^{\delta+1}}{\delta(\delta+1)}-\frac{y}{\delta}-\frac{\delta}{\delta+1},\quad & y \ge 1,\\
	y\ln y-y,\quad & 0< y < 1.
	\end{cases}
	$$
	Then $ I(\cdot) \in \mathcal{C}^{2}((0,\, +\infty))$. By utilizing It\^o's formula on $ I(Y^{t,\, x}) $, and calculating the expectation, we obtain $\forall p >0, $
	$$\mathbb{E}\left[\left( \int_{0}^{T}\frac{\lvert Z_{s}^{t,\, x}\rvert^{2}}{Y_{s}^{t,\, x}}ds \right)^{p}\right] 
	\le C\left\{ 1 + \mathbb{E}\left[ \sup_{0 \le s \le T} \left( Y_{s}^{t,\, x} \right)^{2p(1+\delta)}\right]+ \mathbb{E}\left[\left( \int_{0}^{T}\lvert Z_{s}^{t,\, x}\rvert^{2} ds\right)^{p(1+\delta)/2}\right]\right\}.$$	
	Thus \myeqref{eq22} is now evident from the Vitali convergence theorem.
	It can be easily checked that the integrable condition \myeqref{eq13} still holds. We obtain the continuity of $ u $ through \sref{Theorem}{stability}. A routine computation can establish the growth of $ u $ due to (A.2) and \sref{Corollary}{corunique}. If we can verify that $ u $ is a viscosity subsolution, then the theorem follows immediately.
	Our proof is in the spirit of the touching property (see \citep{Kobylanski2000}). For any fixed smooth function $ \varpi $, we assume $ \varpi-u $ reaches a local minimum at the point $ (s_{0},\, x_{0})$. We can suppose, without loss of generality, that for all $ (s,\, x) \in [0,\,T] \times \mathbb{R}^{n} $, $\varpi(s,\, x) \ge u(s,\,x) $ and $ u(s_{0},\, x_{0}) = \varpi(s_{0},\, x_{0}) $. The aim is to prove 
	$$\partial_{s} \varpi(s_{0},\, x_{0})+\mathscr{L} \varpi(s_{0},\, x_{0})+f\left(s_{0},\, x_{0},\, \varpi(s_{0},\, x_{0}),\, (\left(  \partial_{x} \varpi\right) ^{\top}\sigma)(s_{0},\, x_{0})\right)\ge 0.  $$
	For ease of notation, we treat $ (Y,\,Z) $ as $ \left( Y^{s_{0},\, x_{0}},\,Z^{s_{0},\, x_{0}}\right)  $. 
	From \myeqref{eq20} and It\^o's formula we obtain that
	$$\begin{aligned}
	-d Y_{s} &=  f(t,\, X_{s}^{s_{0},\, x_{0}},\, Y_{s},\,Z_{s}) ds - Z_{s} dW_{s}, \\
	d \varpi(s,\, X^{s_{0},\, x_{0}}_{s}) &= \left(\partial_{s} \varpi(s,\, X^{s_{0},\, x_{0}}_{s}) +\mathscr{L} \varpi(s,\, X^{s_{0},\, x_{0}}_{s})  \right) d s + (\left(  \partial_{x} \varpi\right) ^{\top}\sigma)(s,\, X^{s_{0},\, x_{0}}_{s})d W_{s}.
	\end{aligned}$$
	Note that $ u(s,\,  X^{s_{0},\, x_{0}}_{s}) = Y_{s},$ hence $ \varpi(s,\,  X^{s_{0},\, x_{0}}_{s}) \ge Y_{s}. $ In view of touching property, we have for all $ s \in [0,\,T],$ \as,
	$$\begin{aligned}
	\mathds{1}_{\varpi\left(s,\,  X^{s_{0},\, x_{0}}_{s}\right) = Y_{s} }   \Big(\left(\partial_{s}\varpi + \mathscr{L} \varpi \right)(s,\, X^{s_{0},\, x_{0}}_{s})+ f(s,\, X_{s}^{s_{0},\, x_{0}},\, Y_{s},\,Z_{s}) \Big) &\ge 0, \\ 
	\mathds{1}_{\varpi\left(s,\,  X^{s_{0},\, x_{0}}_{s}\right) = Y_{s} }  \lvert (\left(  \partial_{x} \varpi\right) ^{\top}\sigma)(s,\, X^{s_{0},\, x_{0}}_{s}) -Z_{s}\rvert &=0.\end{aligned}$$
	Due to $ \varpi(s_{0},\,  X^{s_{0},\, x_{0}}_{s_{0}}) = Y_{s_{0}} $, we infer from the second equation that
	$$ Z_{s_{0}} =(\left(  \partial_{x} \varpi\right) ^{\top}\sigma)(s_{0},\,X^{s_{0},\, x_{0}}_{s_{0}}).$$ Noting that $  \Psi(x)= u(T,\,x) = Y_{T}^{T,\,x}  $, it follows that $ u $ is a viscosity subsolution to \myeqref{eq17}.  
\end{proof} 
We further show the uniqueness of viscosity solutions to \myeqref{eq17}. Let us consider the following additional assumption.\\
\textbf{(A.3)} For any given $ \tau >0 $, there exists a function $ m_{\tau}:\left.\left[0,\, +\infty \right. \right) \to \left.\left[0,\, +\infty \right. \right) $ satisfying
$ \lim\limits_{x \to 0}m_{\tau}(x)= 0 $. Moreover, assume that for any $\lvert x^{1} \rvert,\, \lvert x^{2} \rvert,\, \lvert y \rvert \leq \tau, \, z \in \mathbb{R}^{1 \times d},$ and $  0 \leq s \leq T,$
$$\begin{aligned}
\big|  f(s,\, x^{1},\, y,\, z)-f(s,\, x^{2},\, y,\, z) \big|  \le m_{\tau}\bigg( \lvert x^{1} - x^{2} \rvert\left(\lvert z \rvert^{2} +1\right)\bigg).
\end{aligned}$$
\begin{theorem}(Uniqueness)\label{unique}
	PDE \myeqref{eq17} has at most one viscosity solution which satisfies \myeqref{eq21} provided that (A.1), (A.2), and (A.3) are met. 
	In particular, for PDE \myeqref{eq17} the unique viscosity solution fulfilling \myeqref{eq21} is $ u $ defined in \myeqref{eq19}. 
\end{theorem}	
Our proof of \sref{Theorem}{unique} is in the spirit of \citep{barles1997}. 
Let us first present the following lemmas which utilize the same notation in the proof of Lemmas 3.7 and 3.8 in \citep{barles1997}.
\begin{lemma}\label{uniquenessofpde}
	Let $ u $ and $ \kappa $ be two viscosity solutions of the PDE \myeqref{eq17} satisfying \myeqref{eq21}, then $ h :=u-\kappa $ is a viscosity subsolution of
\begin{equation}\label{eq23}
	\left( -\partial_{t} h - \mathscr{L} h\right) h -kh^{2}-k\lvert  \left( \partial_{x}h\right)  ^{\top}\sigma \rvert \lvert h\rvert -\delta \lvert \left(  \partial_{x} h\right) ^{\top}\sigma\rvert^{2 }=0,\quad (t,\,x) \in \left[0,\right.\left. T\right)\times  \mathbb{R}^{n},
	\end{equation}
	where $ k $ and $ \delta $ are constants in (A.2).
\end{lemma}
\begin{proof}[Proof.]
	Let $ (s_{0}, \, w_{0}) \in \left[0,\right.\left. T\right) \times \mathbb{R}^{n} $ as well as $\varphi \in \mathcal{C}^{1,2}\left( \left[ 0,\,T\right] \times \mathbb{R}^{n}\right) $ satisfy that there is a solid neighborhood  of $ (s_{0}, \, w_{0}) $ denoted by $ U(s_{0}, \, w_{0}) $ such that $ 0 = h(s_{0}, \, w_{0})-\varphi(s_{0}, \, w_{0}) > h(s, \, w)-\varphi(s, \, w) $ for each $ (s, \, w)\in U(s_{0}, \, w_{0})\backslash \left\{(s_{0}, \, w_{0})  \right\}$.
	Take a compact set $ \mathcal{T}\times B \subset U(s_{0}, \, w_{0})$ such that $ (s_{0}, \, w_{0}) \in \mathcal{T}\times B $.
	Define a function $ \psi_{\varepsilon}: (0,\,T)\times \mathbb{R}^{n}\times \mathbb{R}^{n} \to \mathbb{R}$,
	$$ \psi_{\varepsilon}(s,\,w,\,z):= u(s,\,w)-\kappa(s,\,z)-\frac{\lvert w - z\rvert^{2}}{2\varepsilon}-\varphi(s,\,w),$$
	in which $ \varepsilon > 0 $ tends to zero. Owing to the continuity of $  \psi_{\varepsilon} $,  the maximum of $ \psi_{\varepsilon} $ on $\mathcal{T}\times B \times B$ is obtained at a point $ (s_{\varepsilon},\,w_{\varepsilon},\,z_{\varepsilon}) \in   \mathcal{T}\times B \times B$. 
	It is not hard to derive that
	\begin{itemize}
		\item[(i)] $ (s_{\varepsilon},\,w_{\varepsilon},\,z_{\varepsilon}) \to (s_{0}, \, w_{0},\, w_{0}) $ as $ \varepsilon \to 0$.
		\item[(ii)]	$\lim\limits_{\varepsilon \to 0} \dfrac{\lvert w_{\varepsilon}- z_{\varepsilon}\rvert^{2}}{\varepsilon}=0. $
	\end{itemize}
	Indeed, since
	$$u(s_{\varepsilon},\,w_{\varepsilon})-\kappa(s_{\varepsilon},\,z_{\varepsilon})-\frac{\lvert w_{\varepsilon} - z_{\varepsilon}\rvert^{2}}{2\varepsilon}-\varphi(s_{\varepsilon},\,w_{\varepsilon}) \geq  u(s_{0},\,w_{0})-\kappa(s_{0},\,w_{0})-\varphi(s_{0},\,w_{0}), $$
	it follows that
	$\lvert w_{\varepsilon} - z_{\varepsilon}\rvert^{2} \leq 4\varepsilon \max\limits_{(s,\,w,\,z)\in\mathcal{T}\times B \times B} \lvert u(s,\,w)-\kappa(s,\,z)-\varphi(s,\,w) \rvert . $
	Sending $ \varepsilon \to 0, $ we obtain that $ w_{\varepsilon} $ and $ z_{\varepsilon} $ converge to the same point. Let us assume  $ (s_{\varepsilon},\,w_{\varepsilon},\,z_{\varepsilon}) $ converges to $ (\tilde{s}, \, \tilde{w},\, \tilde{w}) \in \mathcal{T}\times B \times B$ as $ \varepsilon \to 0$. Then we have
	$$u(\tilde{s},\,\tilde{w})-\kappa(\tilde{s},\,\tilde{w})-\lim\limits_{\varepsilon \to 0}\frac{\lvert w_{\varepsilon} - z_{\varepsilon}\rvert^{2}}{2\varepsilon}-\varphi(\tilde{s},\,\tilde{w})\geq u(s_{0},\,w_{0})-\kappa(s_{0},\,w_{0})-\varphi(s_{0},\,w_{0}).$$
	Note the fact that $ h-\phi  $ achieves its strict maximum in $ U(s_{0}, \, w_{0})$ at $ (s_{0},\,w_{0}) $, the previous inequality implies both (i) and (ii) hold.

	It follows from Theorem 8.3 in	\citep{Crandall1992} that there exist two triplets $ (0, \, p,\, X) \in \overline{D}^{2,+}(u-\varphi)(s_{\varepsilon},\,w_{\varepsilon}) $ and $ (0, \, -p,\, -Z) \in \overline{D}^{2,+}(-\kappa)(s_{\varepsilon},\,z_{\varepsilon}) $ such that
	\begin{itemize}
		\item[(1)] $ p =\dfrac{ w_{\varepsilon}- z_{\varepsilon} }{\varepsilon} $ ;
		\item[(2)]	$ -\dfrac{3}{\varepsilon} \begin{pmatrix}
		I	&  0\\
		0	&  I\\
		\end{pmatrix}  \leq  \begin{pmatrix}
		X	&  0\\
		0	& -Z \\
		\end{pmatrix}	\leq
		\dfrac{3}{\varepsilon} \begin{pmatrix}
		I	&  -I\\
		-I	&  I\\
		\end{pmatrix}.
		$
	\end{itemize}
	It follows that $ (\varphi_{s}(s_{\varepsilon},\,w_{\varepsilon}), \, p+\partial_{w} \varphi(s_{\varepsilon},\,w_{\varepsilon}),\, X+D^{2}\varphi(s_{\varepsilon},\,w_{\varepsilon})) \in \overline{D}^{2,+}(u-\varphi)(s_{\varepsilon},\,w_{\varepsilon}) $ and $ (0, \, p,\, Z) \in \overline{D}^{2,-}\kappa(s_{\varepsilon},\,z_{\varepsilon}) $. Since $ u $ and $ \kappa $ are respectively a viscosity subsolution and a  viscosity supersolution of \myeqref{eq17} we deduce that
	$$\begin{aligned}
	-&\varphi_{s}(s_{\varepsilon},\,w_{\varepsilon})-\frac{1}{2}tr\left( \left( \sigma\sigma^{\top}\right) ( s_{\varepsilon},\,w_{\varepsilon})X\right) -B^{\top}(s_{\varepsilon},\,w_{\varepsilon})p\\-&\mathscr{L}\varphi(s_{\varepsilon},\,w_{\varepsilon})-f\left(s_{\varepsilon},\,w_{\varepsilon},\,u(s_{\varepsilon},\,w_{\varepsilon}),\,(p+\partial_{w} \varphi(s_{\varepsilon},\,w_{\varepsilon})^{\top}\sigma(s_{\varepsilon},\,w_{\varepsilon})) \right) \leq 0		
	\end{aligned}$$
	and
	$$-\frac{1}{2}tr\left( \left( \sigma\sigma^{\top}\right) (s_{\varepsilon},\,z_{\varepsilon})Z\right) -B^{\top}(s_{\varepsilon},\,z_{\varepsilon})p-f\left(s_{\varepsilon},\,z_{\varepsilon},\,\kappa(s_{\varepsilon},\,z_{\varepsilon}),\,p^{\top} \sigma(s_{\varepsilon},\,z_{\varepsilon})\right) \geq 0.		
	$$
	
	If $ (e_{1},\,e_{2},\,...,\,e_{n}) $ is an orthonormal basis of $ \mathbb{R}^{n} $, 
	$$	\begin{aligned}
	&tr\left( \left( \sigma\sigma\right) ^{\top}(s_{\varepsilon},\,w_{\varepsilon})X\right) -tr\left( \left( \sigma\sigma^{\top}\right) (s_{\varepsilon},\,z_{\varepsilon})Z\right) 
	\\	= &\,\sum_{i=1}^{n}\left[ \left( \sigma(s_{\varepsilon},\,w_{\varepsilon})e_{i}\right)^{\top}\left(X \sigma(s_{\varepsilon},\,w_{\varepsilon})e_{i}\right) - \left( \sigma(s_{\varepsilon},\,z_{\varepsilon})e_{i}\right)^{\top}\left(Z \sigma(s_{\varepsilon},\,z_{\varepsilon})e_{i}\right)\right]
	\leq \,	\dfrac{3J^{2}}{\varepsilon} \lvert w_{\varepsilon}-z_{\varepsilon}\rvert^{2}.
	\end{aligned}	
	$$	
	It follows from assumptions (A.1) and (A.3) that 
	$ p^{\top}\left(  B(s_{\varepsilon},\,w_{\varepsilon})-B(s_{\varepsilon},\,z_{\varepsilon})\right) \leq \dfrac{J}{\varepsilon} \lvert w_{\varepsilon}-z_{\varepsilon}\rvert^{2} $ and
	there exists an enough large constant $ R $ such that $\lvert w_{\varepsilon}\rvert,\,\lvert z_{\varepsilon}\rvert,\, u(s_{\varepsilon},\,w_{\varepsilon}),\, \kappa(s_{\varepsilon},\,z_{\varepsilon}) \leq R.$
	
	If $ u(s_{\varepsilon},\,w_{\varepsilon}) > \kappa(s_{\varepsilon},\,z_{\varepsilon}) $, subtracting the viscosity inequalities for $ u $ and $ \kappa$ yields that for any fixed $0<\theta<1 $:
	$$\begin{aligned}
	&-\varphi_{s}(s_{\varepsilon},\,w_{\varepsilon})-\mathscr{L}\varphi(s_{\varepsilon},\,w_{\varepsilon})  \\
	\leq \,&\frac{1}{2}tr\left( \left( \sigma\sigma^{\top}\right) (s_{\varepsilon},\,w_{\varepsilon})X\right) - \frac{1}{2}tr\left( \left( \sigma\sigma^{\top}\right) (s_{\varepsilon},\,z_{\varepsilon})Z\right) 
	+ p^{\top} \left( B(s_{\varepsilon},\,w_{\varepsilon})-B(s_{\varepsilon},\,z_{\varepsilon})\right)  \\+& f\left(s_{\varepsilon},\,w_{\varepsilon},\,u(s_{\varepsilon},\,w_{\varepsilon}),\,(p+\partial_{w} \varphi(s_{\varepsilon},\,w_{\varepsilon}))^{\top}\sigma(s_{\varepsilon},\,w_{\varepsilon}) \right) \\-&f\left(s_{\varepsilon},\,z_{\varepsilon},\,u(s_{\varepsilon},\,w_{\varepsilon}),\,(p+\partial_{w} \varphi(s_{\varepsilon},\,w_{\varepsilon}))^{\top}\sigma(s_{\varepsilon},\,w_{\varepsilon}) \right) \\   +&f\left(s_{\varepsilon},\,z_{\varepsilon},\,u(s_{\varepsilon},\,w_{\varepsilon}),\,(p+\partial_{w} \varphi(s_{\varepsilon},\,w_{\varepsilon}))^{\top}\sigma(s_{\varepsilon},\,w_{\varepsilon})\right) 
	-f\left(s_{\varepsilon},\,z_{\varepsilon},\,\kappa(s_{\varepsilon},\,z_{\varepsilon}),\,p^{\top}\sigma(s_{\varepsilon},\,z_{\varepsilon}) \right)\\
	\leq \,&
	\dfrac{3J^{2}+2J}{2\varepsilon} \lvert w_{\varepsilon}-z_{\varepsilon}\rvert^{2}+ m_{R}\bigg( \lvert w_{\varepsilon} - z_{\varepsilon}\rvert\left(1+\lvert \left( p+\partial_{w} \varphi(s_{\varepsilon},\,w_{\varepsilon})\right) ^{\top}\sigma(s_{\varepsilon},\,w_{\varepsilon}) \rvert^{2} \right)\bigg)\\ 
	+&
	(1-\theta)f\left(s_{\varepsilon},\,z_{\varepsilon},\,\dfrac{u(s_{\varepsilon},\,w_{\varepsilon})-\theta \kappa(s_{\varepsilon},\,z_{\varepsilon})}{1-\theta},\,\dfrac{(p+\partial_{w} \varphi(s_{\varepsilon},\,w_{\varepsilon}))^{\top}\sigma(s_{\varepsilon},\,w_{\varepsilon}) -\theta p^{\top}\sigma(s_{\varepsilon},\,z_{\varepsilon})}{1-\theta} \right).
	\end{aligned}	$$
	Multiplying both sides by $ u(s_{\varepsilon},\,w_{\varepsilon})-\theta \kappa(s_{\varepsilon},\,z_{\varepsilon}) $, we have
	$$\begin{aligned}
	&\bigg( -\varphi_{s}(s_{\varepsilon},\,w_{\varepsilon})-\mathscr{L}\varphi(s_{\varepsilon},\,w_{\varepsilon}) \bigg) \bigg(u(s_{\varepsilon},\,w_{\varepsilon})-\theta \kappa(s_{\varepsilon},\,z_{\varepsilon}) \bigg)  \\
	\leq\,&  \left[
	\dfrac{3J^{2}+2J}{2\varepsilon} \lvert w_{\varepsilon}-z_{\varepsilon}\rvert^{2}
	+ m_{R}\bigg( \lvert w_{\varepsilon} - z_{\varepsilon}\rvert\left(1+\lvert \left( p+\partial_{w} \varphi(s_{\varepsilon},\,w_{\varepsilon})\right) ^{\top}\sigma(s_{\varepsilon},\,w_{\varepsilon}) \rvert^{2} \right)\bigg)\right.
	\\+&	k(1-\theta)(1+\lvert  z_{\varepsilon}\rvert^{q})	+ k\left(u(s_{\varepsilon},\,w_{\varepsilon})-\theta \kappa(s_{\varepsilon},\,z_{\varepsilon})\right)+
	k\lvert  \left(\left(  \partial_{w} \varphi\right) ^{\top}\sigma\right) (s_{\varepsilon},\,w_{\varepsilon})\rvert  \\+&\left.	
	k	\dfrac{\lvert w_{\varepsilon}-z_{\varepsilon}\rvert}{\varepsilon}\lvert   \sigma(s_{\varepsilon},\,w_{\varepsilon}) -\theta\sigma(s_{\varepsilon},\,z_{\varepsilon})\rvert \right] \bigg(u(s_{\varepsilon},\,w_{\varepsilon})-\theta \kappa(s_{\varepsilon},\,z_{\varepsilon})\bigg)\\+&\delta\lvert  \left(\left(  \partial_{w} \varphi\right) ^{\top}\sigma\right)(s_{\varepsilon},\,w_{\varepsilon})\rvert ^{2}+
	\delta  \dfrac{\lvert w_{\varepsilon}-z_{\varepsilon}\rvert^{2}}{\varepsilon^{2}} \lvert  \sigma(s_{\varepsilon},\,w_{\varepsilon}) -\theta\sigma(s_{\varepsilon},\,z_{\varepsilon})\rvert ^{2}.	
	\end{aligned}	$$
	If $ u(s_{\varepsilon},\,w_{\varepsilon}) < \kappa(s_{\varepsilon},\,z_{\varepsilon}) $,	noting that the fact that $ u $ is a viscosity supersolution and $ \kappa $ is a viscosity subsolution, a similar procedure yields that
	$$\begin{aligned}
	&\left( -\varphi_{s}(s_{\varepsilon},\,w_{\varepsilon})-\mathscr{L}\varphi(s_{\varepsilon},\,w_{\varepsilon}) \right) \bigg(\theta u(s_{\varepsilon},\,w_{\varepsilon})- \kappa(s_{\varepsilon},\,z_{\varepsilon}) \bigg)  \\
	\leq& \left[
	\dfrac{3J^{2}+2J}{2\varepsilon} \lvert w_{\varepsilon}-z_{\varepsilon}\rvert^{2}+ m_{R}\bigg( \lvert w_{\varepsilon} - z_{\varepsilon}\rvert\left(1+\lvert \left( p+\partial_{w} \varphi(s_{\varepsilon},\,w_{\varepsilon})\right) ^{\top}\sigma(s_{\varepsilon},\,w_{\varepsilon}) \rvert^{2} \right)\bigg)\right.
	\\	+&	k(1-\theta)(1+\lvert  z_{\varepsilon}\rvert^{q})	+  k\left(\kappa(s_{\varepsilon},\,z_{\varepsilon})-\theta u(s_{\varepsilon},\,w_{\varepsilon}) \right)+
	k\lvert  \left(\left(  \partial_{w} \varphi\right) ^{\top}\sigma\right)(s_{\varepsilon},\,w_{\varepsilon})\rvert  \\+&\left.	k	\dfrac{\lvert w_{\varepsilon}-z_{\varepsilon}\rvert}{\varepsilon}\lvert  \sigma(s_{\varepsilon},\,z_{\varepsilon})-\theta \sigma(s_{\varepsilon},\,w_{\varepsilon}) \rvert \right] \bigg(h(s_{\varepsilon},\,z_{\varepsilon})- \theta u(s_{\varepsilon},\,w_{\varepsilon}) \bigg)\\+&\delta\lvert  \left(\left(  \partial_{w} \varphi\right) ^{\top}\sigma\right)(s_{\varepsilon},\,w_{\varepsilon})\rvert ^{2}+
	\delta  \dfrac{\lvert w_{\varepsilon}-z_{\varepsilon}\rvert^{2}}{\varepsilon^{2}} \lvert \sigma(s_{\varepsilon},\,z_{\varepsilon})-\theta \sigma(s_{\varepsilon},\,w_{\varepsilon}) \rvert ^{2}.		
	\end{aligned}	$$
	
	Combining these two cases and letting $ \theta \to 1 $, we can derive that
	$$\begin{aligned}
	&\left( -\varphi_{s}(s_{\varepsilon},\,w_{\varepsilon})-\mathscr{L}\varphi(s_{\varepsilon},\,w_{\varepsilon}) \right) \bigg( u(s_{\varepsilon},\,w_{\varepsilon})- \kappa(s_{\varepsilon},\,z_{\varepsilon}) \bigg) \\
	\leq& \left\{
	\dfrac{3J^{2}+2J}{2\varepsilon} \lvert w_{\varepsilon}-z_{\varepsilon}\rvert^{2}+ m_{R}\bigg( \lvert w_{\varepsilon} - z_{\varepsilon}\rvert\left(1+\lvert \left( p+\partial_{w} \varphi(s_{\varepsilon},\,w_{\varepsilon})\right) ^{\top}\sigma(s_{\varepsilon},\,w_{\varepsilon}) \rvert^{2} \right)\bigg)
	\right.	
	\\ 	+& k\lvert u(s_{\varepsilon},\,w_{\varepsilon})- \kappa(s_{\varepsilon},\,z_{\varepsilon}) \rvert +
	k\lvert  \left(\left(  \partial_{w} \varphi\right) ^{\top}\sigma\right)(s_{\varepsilon},\,w_{\varepsilon})\rvert  \\+&\left.	k	\dfrac{\lvert w_{\varepsilon}-z_{\varepsilon}\rvert}{\varepsilon}\lvert  \sigma(s_{\varepsilon},\,w_{\varepsilon}) -  \sigma(s_{\varepsilon},\,z_{\varepsilon}) \rvert \right\} \bigg|u(s_{\varepsilon},\,w_{\varepsilon})-\kappa(s_{\varepsilon},\,z_{\varepsilon}) \bigg|\\+&\delta\lvert  \left(\left(  \partial_{w} \varphi\right) ^{\top}\sigma\right)(s_{\varepsilon},\,w_{\varepsilon})\rvert ^{2}+
	\delta  \dfrac{\lvert w_{\varepsilon}-z_{\varepsilon}\rvert^{2}}{\varepsilon^{2}} \lvert \sigma(s_{\varepsilon},\,w_{\varepsilon})- \sigma(s_{\varepsilon},\,z_{\varepsilon}) \rvert ^{2}.		
	\end{aligned}	$$
	It's obvious the above inequality still holds if $ u(s_{\varepsilon},\,w_{\varepsilon}) =\kappa(s_{\varepsilon},\,z_{\varepsilon}) $. By sending $  \varepsilon \to 0$ we complete the proof.	
\end{proof}

The following lemma aims to establish a supersolution of  \myeqref{eq23}.
\begin{lemma}\label{supersolution}
	For any $ A>0 $ and sufficiently large $ C$, the function
	$$ \mathcal{X}(s,\,w):=e^{\left( A+ C(T-s)\right) \left( 1+\lvert w \rvert ^{2} \right) },\quad(s,\,w)\in [s_{1},\,T]\times \mathbb{R}^{n},$$
	where $s_{1}:=T-\frac{A}{C} ,$ satisfies for all $(s,\,w)\in [s_{1},\,T]\times \mathbb{R}^{n}, $
	$$-\partial_{s}\mathcal{X}(s,\,w)-\mathscr{L}\mathcal{X}(s,\,w)-k\mathcal{X}(s,\,w)-k\lvert \left( \sigma^{\top}\partial_{w}\mathcal{X}\right) (s,\,w) \rvert -\delta\dfrac{\lvert \left( \sigma^{\top}\partial_{w}\mathcal{X}\right) (s,\,w) \rvert ^{2}}{\mathcal{X}(s,\,w)}>0.$$
\end{lemma}	
\begin{proof}[Proof.]
	Given $ (s,\,w)\in [s_{1},\,T]\times \mathbb{R}^{n} $, it is clear that $\partial_{s}\mathcal{X}(s,\,w)=-C\mathcal{X}(s,\,w)(1+\lvert w \rvert ^{2})$, $
	\partial_{w}\mathcal{X}(s,\,w)=2\left(  C(T-s)+A\right) \mathcal{X}(s,\,w)w
	$, and $D^{2}\mathcal{X}(s,\,w)=2\left(  C(T-s)+A\right)\mathcal{X}(s,\,w)I+4\left(  C(T-s)+A\right)^{2}\mathcal{X}(s,\,w)ww^{\top}.
	$ Consequently, $$\lvert \partial_{w}\mathcal{X}(s,\,w)\rvert \leq 4A \mathcal{X}(s,\,w)\lvert  w\rvert  \text{ and } 
	\lvert  D^{2}\mathcal{X}(s,\,w)\rvert  =4A\mathcal{X}(s,\,w)+16 A^{2}\mathcal{X}(s,\,w)\lvert  w\rvert ^{2}.$$	
	Finally, we conclude that
	$$\begin{aligned}
	&-\partial_{s}\mathcal{X}(s,\,w)-\mathscr{L}\mathcal{X}(s,\,w)-k\mathcal{X}(s,\,w)-k\lvert \left( \sigma^{\top}\partial_{w}\mathcal{X}\right) (s,\,w) \rvert -\delta\frac{\lvert \left( \sigma^{\top}\partial_{w}\mathcal{X}\right) (s,\,w) \rvert ^{2}}{\mathcal{X}(s,\,w)}\\
	\geq	 \,
	&\mathcal{X}(s,\,w)\left[ C(1+\lvert w \rvert ^{2})-J^{2}\left(2A+8 A^{2}\lvert  w\rvert ^{2} \right)- 4AJ(1+k+\lvert  w\rvert ) \lvert  w\rvert -k -\delta 16 A^{2}J^{2}\lvert  w\rvert ^{2}\right] . \end{aligned}$$	
	We complete the proof by taking $ C $ large enough. 
\end{proof}
\begin{proof}[Proof of Theorem 6.]
	For any fixed $ \alpha>0$, define $N(s,\,w):=\left(h(s,\,w)-\alpha \mathcal{X}(s,\,w) \right)e^{ks},\, (s,\,w) \in \left[s_{1},\,T\right]\times \mathbb{R}^{n}$. Since for some $ A>0 $, it holds that for each $ s \in [s_{1},\,T], $
	$$\lim\limits_{\lvert w\rvert \to \infty}\lvert h(s,\,w)\rvert e^{-A\lvert w\rvert ^{2}}=0,$$
	there exists a compact set $ E \in \mathbb{R}^{n}  $ satisfying that
	$N(s,\,w) < 0, \forall (s,\,w) \in [s_{1},\,T] \times E^{c}$.
	Since $ N $ is continuous, $ \max\limits_{(s,\,w)\in[s_{1},\,T]\times E}N(s,\,w) $ is achieved at some point $ (\bar{s},\,\bar{w})$. Suppose that $ N(\bar{s},\,\bar{w})>0 $. Then the maximum of $ N $ on $[s_{1},\,T]\times \mathbb{R}^{n} $ is obtained at $ (\bar{s},\,\bar{w}) $ and $ \bar{s}<T. $
	Let us introduce a function
	$$\varphi(s,\,w):=\alpha \mathcal{X}(s,\,w) + \left(h(\bar{s},\,\bar{w})-\alpha \mathcal{X}(\bar{s},\,\bar{w}) \right)e^{k\left( \bar{s}-s\right) }, (s,\,w) \in[s_{1},\,T]\times \mathbb{R}^{n}.	$$	
	Then $ h - \varphi$ obtains the global  maximum on $ [s_{1},\,T]\times \mathbb{R}^{n}$ at $ (\bar{s},\,\bar{w}) $. According to \sref{Lemma}{uniquenessofpde}, 
	$$	\left( \left(-\partial_{s} \varphi - \mathscr{L}  \varphi \right)  \varphi\right) (\bar{s},\,\bar{w}) -k \varphi^{2} (\bar{s},\,\bar{w})-k\lvert \left( \sigma^{\top} \partial_{w}  \varphi \right) (\bar{s},\,\bar{w}) \rvert  \varphi (\bar{s},\,\bar{w})  -\delta \lvert \left( \sigma^{\top}   \partial_{w}  \varphi \right) (\bar{s},\,\bar{w})\rvert^{2 }\leq 0,$$
	while it follows from the definition of $ \varphi $ and \sref{Lemma}{supersolution} that
	$$\begin{aligned}
	&\left(\left( -\partial_{s} \varphi - \mathscr{L}  \varphi \right)  \varphi\right) (\bar{s},\,\bar{w}) -k \varphi^{2} (\bar{s},\,\bar{w})-k\lvert \left( \sigma^{\top} \partial_{w}  \varphi \right) (\bar{s},\,\bar{w}) \rvert  \varphi (\bar{s},\,\bar{w}) -\delta \lvert \left( \sigma^{\top}   \partial_{w}  \varphi \right) (\bar{s},\,\bar{w})\rvert^{2 }\\
	\geq &\,\alpha h(\bar{s},\,\bar{w})\left[  \left(-\partial_{s}\mathcal{X}-\mathscr{L}\mathcal{X}-k\mathcal{X}\right)(\bar{s},\,\bar{w})-k\lvert \left( \sigma^{\top}\partial_{w}\mathcal{X}\right) (\bar{s},\,\bar{w}) \rvert -\delta\dfrac{\lvert \left( \sigma^{\top}\partial_{w}\mathcal{X}\right) (\bar{s},\,\bar{w}) \rvert ^{2}}{\mathcal{X}(\bar{s},\,\bar{w})}\right] \\>&\,0,
	\end{aligned}$$
	which is a contradiction with the previous inequality. Therefore 
	$ h(s,\,w)\leq \alpha \mathcal{X}(s,\,w),\forall (s,\,w) \in  [s_{1},\,T]\times \mathbb{R}^{n}.$
	Sending $ \alpha \to 0 $ we have that $ u\leq \kappa $ on $ [s_{1},\,T]\times \mathbb{R}^{n} $.
	
	Setting $ s_{2}:=\left( s_{1}-\frac{A}{C}\right) ^{+} $, we can obtain analogously that $ u\leq \kappa $ on $ [s_{2},\,s_{1}]\times \mathbb{R}^{n} $. Step by step, we prove that $ u\leq \kappa $ on $ [0,\,T]\times \mathbb{R}^{n} $. Ultimately, interchanging $ u $ and $ \kappa $ and  we can similarly show that $ \kappa\leq u $ on $ [0,\,T]\times \mathbb{R}^{n} $. This completes the proof.\end{proof}

\section{Applications in Finance}\label{sec5}
\subsection{Linking Robust Control to Stochastical Differential Utility}\label{subsec5}

Since \citep{HS1995} pioneered the integration of robust control theory from engineering into economics, an increasing number of studies have investigated the impact of model uncertainty on portfolio choices and asset pricing (see \citep{2002Robustness, Skiadas2003, Maenhout2004, Yang2019, Pu2021}). In this paper, we exploit the methodology developed in \citep{Skiadas2003}, which showed the connection between a robust control criterion proposed by \citep{2002Robustness} and a form of the stochastic differential utility(SDU) of \citep{DuffieEpstein1992} without relying on any underlying dynamics. 
Let us consider a progressively measurable process $F:[0,\, T] \times  \Omega \times \mathbb{R} \mapsto \mathbb{R}$, a continuous function $\psi: (0,\,+\infty)\mapsto (0,\,+\infty)$ together with a terminal value $ \xi $. $ X $ represents the set of $ N \in \mathcal{L}^{2}(\mathbb{R}^{d\times1}) $ fulfilling $ \mathscr{E}(N^{\top}\circ W)  $ is a $ \mathbb{P} $-martingale. For each $ N \in X,\, E^{N} $ denotes the expectation under the probability measure $ P^{N} $ defined by
$$  \frac{dP^{N}}{d\mathbb{P}}\Big|_{\mathcal{F}_{t}} = \mathscr{E}(N^{\top} \circ W)_{t}.$$
We consider the following robust control criterion:
\begin{equation}\label{eq24}
\hat{V}_{t}=\text { ess } \inf \bigg\{V_{t}^{x}: x \in \hat{X} \bigg\},
\end{equation}
where 
\begin{equation}\label{eq25}
V_{t}^{x}=E^{x}\left[ \xi + \int_{t}^{T}\left(  F(s,\, V_{s}^{x})+\frac{\psi(-V_{s}^{x})}{2}|x_{s}|^{2} \right)  d s \big| \mathcal{F}_{t}\right],
\end{equation}
and 
$$ \hat{X} := \left\{ x \in X: \text{ a.s. } P^{x},\,  V^{x} < 0 \text{ is bounded }  \right\}. $$

In the framework of \citep{Skiadas2003}, $ \hat{V} $ equals to the SDU $ V $ that satisfies the BSDE
\begin{equation}\label{eq26}
V_{t} = \xi + \int_{t}^{T} F(r,\, V_{r}) -\dfrac{\lvert M_{r}\rvert ^{2}}{2\psi(-V_{r})}  dr -\int_{t}^{T} M_{r} dW_{r},\ t \in[0,\, T] 
\end{equation}
provided that in a properly defined space, BSDE \myeqref{eq26} has a unique solution $ (V,\, M) $.
The following proposition fills this gap according to our existence and uniqueness results.

\begin{proposition}	Assume $ \xi $ and the generator of \myeqref{eq26} satisfy the following conditions.\label{propositon5.1}
	\begin{itemize}
		\item[ $ \small {\bullet}$] $ \xi $ is bounded and $\xi < - D $ for some constant $ D>0 $. 
		\item[ $ \small {\bullet}$] Given any fixed $ (s,\, \omega) $, $ F $ is concave and continuous, and satisfies
		$$\forall y < 0,\quad -\alpha_{s}(\omega)-\beta_{s}(\omega)\phi(-y) \le F(s,\, \omega,\, y) \le 0, $$
		in which $ \alpha $ and $ \beta$ are nonnegative processes such that $ \int_{0}^{T} \alpha_{s} d s $ and $ \int_{0}^{T} \beta_{s} d s$ are bounded and where $\phi,\, \psi:(0,\,+\infty)\mapsto (0,\,+\infty) $ satisfy $\phi \in \mathcal{C}^{1}((0,\,+\infty)) $ is convex,  $ \psi $ is concave, increasing, and continuous, $ \phi / \psi $ is increasing and $\left(\mathds{1}_{x > 0 }/\psi(x) \right) _{x \in \mathbb{R}} $ is locally integrable on $   \mathbb{R}  $.	 
	\end{itemize}
	Then BSDE \myeqref{eq26} admits a unique negative solution $ (V,\, M) $ such that $ V \in \mathcal{S} ^{\infty}( \mathbb{R}) $ and $ M\circ W \in BMO$. 
	With regard to the robust control criterion \myeqref{eq24}, if $\hat{x} = -\frac{1}{\psi(-V)}M^{\top}$, then  $ \hat{x} \in \hat{X} $, and
	$$\hat{V} = V^{\hat{x}} = V.$$
\end{proposition}
\begin{proof}[Proof.]
	Clearly, $ (V,\, M) $ solves \myeqref{eq26} if and only if $ (\bar{V},\, \bar{M}) := (-V,\, -M) $ solves 
	$$\bar{V}_{t} = -\xi + \int_{t}^{T}\left(-F(s,\, -\bar{V}_{s}) +\dfrac{\lvert\bar{M}_{s}\rvert ^{2}}{2\psi(\bar{V}_{s})} \right) ds -\int_{t}^{T} \bar{M}_{s} dW_{s},\ t \in [0\, T],   $$
	which admits a unique solution according to \sref{Section}{sec2} and \ref{sec3}. It remains to prove the last argument. 
	
	We first show that for any $ x \in \hat{X},\, V\leq V^{x}. $ According to the martingale representation theorem, there exists $ M^{x} \in \mathcal{L}^{2}$ fulfilling
\begin{equation}\label{eq27}
	V_{t}^{x}=\xi + \int_{t}^{T} F(s,\, V_{s}^{x}) +  \dfrac{\psi(-V_{s}^{x})}{2} |x_{s}|^{2} d s   -\int_{t}^{T} M^{x}_{s} dW^{x}_{s}.
	\end{equation}
	Note that \myeqref{eq26} can be reformulated as
	$$V_{t} = \xi + \int_{t}^{T} F(s,\, V_{s}) -\dfrac{\lvert M_{s}\rvert ^{2}}{2\psi(-V_{s})} - M_{s}x_{s} ds -\int_{t}^{T} M_{s} dW^{x}_{s},\, t \in[0,\, T]. $$
	Combining dynamics for $V^{x} $ and $V $, 
	we can now derive from Tanaka's formula that for each fixed $0< \theta < 1$,
	$$\begin{aligned}
	&\left(\theta V_{t}-V_{t}^{x}\right)^{+}\\=&(\theta -1)\xi + \int_{t}^{T} \mathds{1}_{\theta V_{s}>V_{s}^{x}}\Big\{\Big[\theta F(s,\, V_{s})-F(s,\, V_{s}^{x})\Big] 
	-\dfrac{\theta \psi(-V_{s})}{2} \lvert M_{s}\dfrac{1}{\psi(-V_{s})} + x_{s} \rvert^{2} \\
	+ &\dfrac{\lvert x_{s} \rvert^{2}}{2} \Big( \theta\psi(-V_{s})-\psi(-V_{s}^{x}) \Big)\Big\} ds   
	- \int_{t}^{T} \mathds{1}_{\theta V_{s}>V_{s}^{x}} \left(\theta M_{s}-M^{x}_{s}\right) dW^{x}_{s}-\frac{1}{2}\int_{t}^{T}dL^{0}_{s}
	\left(\theta V-V^{x}\right)\\
	\leq & (\theta -1)\xi + \int_{t}^{T} \mathds{1}_{\theta V_{s}>V_{s}^{x}} (1-\theta)\left(\alpha_{s}+\beta_{s}\phi(\frac{\theta V_{s}-V_{s}^{x}}{1-\theta})\right)ds -\int_{t}^{T} \mathds{1}_{\theta V_{s}>V_{s}^{x}} \left(\theta M_{s}-M_{s}^{x}\right) dW^{x}_{s}
	.\end{aligned}$$
	To eliminate $ \phi(\frac{\theta V_{s}-V_{s}^{x}}{1-\theta}) $, define 
	$$G(y):=\int_{2}^{2+y}\frac{dr}{\phi(r)},\,\ y> -2. $$
	Obviously, $ G\in \mathcal{C}^{2}((-2,\, +\infty)) $ and $ G$ is strictly increasing. Applying It\^o's formula to
	$$ P_{t} := G^{-1}\left(G\left(\frac{\left(\theta V_{t}-V_{t}^{x}\right)^{+}}{1-\theta}\right)+\int_{0}^{t}\left(\frac{\alpha_{s}}{\phi(1)}+\beta_{s} \right)ds\right) $$ and denoting
	$$Q_{t} := \frac{\phi(2+P_{t})}{\phi(2+\frac{\left(\theta V_{t}-V_{t}^{x}\right)^{+}}{1-\theta})}\mathds{1}_{\theta V_{t}>V_{t}^{x}} \frac{\theta M_{t}-M^{x}_{t}}{1-\theta} $$
	yield that
	$$ P_{t}\leq G^{-1}\left(G\left( -\xi \right)+\int_{0}^{T}\left(\frac{\alpha_{s}}{\phi(1)}+\beta_{s} \right)ds\right) -\int_{t}^{T} Q_{s} dW^{x}_{s}. $$
	Provided that $ V^{x},\, \int_{0}^{T}\alpha_{s}ds $ and $ \int_{0}^{T}\beta_{s}ds $ are bounded, it is apparent from the localization technique and dominated convergence theorem that
	$$ \left(\theta V_{t}-V_{t}^{x}\right)^{+} \leq (1-\theta) E_{t}^{x}\left[G^{-1}\left(G\left( -\xi \right)+\int_{0}^{T}\left(\frac{\alpha_{s}}{\phi(1)}+\beta_{s} \right)ds\right)  \right].  $$
	Sending $\theta \to 1,\, $ we obtain that for any $ x \in  \hat{X}$, $ V  \leq V^{x} $. 
	Now we prove that if $\hat{x} = -\frac{1}{\psi(-V)}M^{\top}$,  then $ V^{\hat{x}} \leq  V$, which can be verified through applying Tanaka's formula to $ \left( \theta V^{\hat{x}}-V\right)^{+} $.
	
	In conclusion, if $ \hat{x} \in \hat{X} $, then $ \hat{V} = V^{\hat{x}} = V  $.	Note that $ V \in \mathcal{S} ^{\infty} $ and $ M\circ W \in BMO$, it is evident that $ \hat{x} \in X.$ 
	Thanks to the boundness of
	$V,\,\xi,\,\int_{0}^{T}\alpha_{s}ds $, and $ \int_{0}^{T}\beta_{s}ds $, we can verify that   $  V^{\hat{x}} $ is bounded a.s. $ P^{\hat{x}} $. This completes the proof. 
\end{proof}
\begin{example}
	As an example, consider the Epstein-Zin type stochastic differential utility used in \citep{Maenhout2004}, which is	
	$$F(t,\, V_{t}^{x}) = \frac{1}{1-\rho}\left\{\frac{c_{t}^{1-\rho}}{((1-\gamma) V_{t}^{x})^{\frac{\gamma-\rho}{1-\gamma}}}-\eta(1-\gamma) V_{t}^{x}\right\}, $$
	where $\rho>1,\, \gamma>1 $ and $ \eta \ge 0.$ 
	
	Suppose we plug this process back into \myeqref{eq25}. It's easy to check \sref{Propositon}{propositon5.1} still holds when $ \gamma <\rho,\, \psi(x) \le \frac{\gamma-1}{\rho-\gamma}x,\,\forall x>0$ and $ \int_{0}^{T}c_{t}^{1-\rho}dt $ is bounded. This result is a slight extension of that in \citep{Maenhout2004}.
\end{example}

\subsection{Certainty Equivalent Based on g-expectation}

We take an example to show that the value process of the BSDE considered in this paper can represent a certainty equivalent of the terminal value based on g-expectation. 
\begin{example}\label{ex2}
	Suppose that $ \gamma $ is a real-valued nonnegative stochastic process, $ g(t,\, z) := \gamma_{t}\lvert z \rvert $ is a generator, and $ \xi < 0$ is a terminal condition. Take a utility function $ u $ defined as 
	$$
	u(y) := 
	e^{2\sqrt{-y}}\left( \frac{1}{2}-\sqrt{-y}\right)-\frac{1}{2},\quad  y<0.	$$
	It is apparent that for $y<0$, $ u'(y)= e^{2\sqrt{-y}},\,u''(y)=  \frac{1}{\sqrt{-y}}e^{2\sqrt{-y}}.$
	If
	$$
	\xi<0,\,\mathbb{E}\left[\exp\Big \{\frac{p}{p-1}\int_{0}^{T} \gamma_{s}^{2} d s\Big\}\left( 1+ e^{4p\sqrt{- \xi}}\left( -\xi\right)^{p}\right) \right]<+\infty \text{ for some } p > 1 , $$
	there is only one solution $(\bar{Y},\, \bar{Z}) $ in $  \mathcal{S}^{2p} \times \mathcal{M}^{p} $ satisfying the BSDE
	$$ \bar{Y}_{t} = u(\xi) + \int_{t}^{T} \left( -\gamma_{s}\lvert \bar{Z}_{s} \rvert \right) ds -\int_{t}^{T} \bar{Z}_{s} dW_{s},\quad t \in [ 0, \, T]. $$
	
	Denote the conditional g-expectation of $ u(\xi) $ under $ \mathcal{F}_{t}$ by $ \varepsilon_{t}^{g}\left[u(\xi)\right]:= \bar{Y}_{t}$ (see more about g-expectation in \citep{1997Backward}). The certainty equivalent of $ \xi $ at time $t \in [0,\, T] $ based on g-expectation is defined as
	$$ C_{t}(\xi) := u^{-1}\left(\varepsilon_{t}^{g}\left[u\left( \xi\right) \right]\right).$$	
	Putting $ Z_{t} = \frac{1}{u'\big( u^{-1}\left(\bar{Y}_{t}\right) \big) } \bar{Z}_{t}$, the pair $ \left( C(\xi),\, Z\right) \in \mathcal{S}^{2p}\times \mathcal{M}^{p}  $ is the solution to the BSDE
	$$ Y_{t} = \xi - \int_{t}^{T}\left( \gamma_{r}\lvert Z_{r} \rvert + \frac{1}{2}\frac{\lvert Z_{r}\rvert ^{2}}{\sqrt{-Y_{r}}} \right)  dr -\int_{t}^{T} Z_{r} dW_{r},\quad   t \in [ 0, \, T],$$
	in which $ \frac{1}{\sqrt{-y}} $ represents the coefficient of absolute risk aversion of $ u $ at level $ y $.
\end{example}

\section{Discussion}

In this article, we establish the existence and comparison theorem for $ L^{p} $ solutions to a class of backward stochastic differential equations (BSDEs) with singular generators, extending the results of \citep{Bahlali2018}. Additionally, we investigate the stability property and the Feynman-Kac formula, and prove the uniqueness of viscosity solutions to the associated partial differential equations (PDEs). More importantly, we demonstrate applications in the context of robust control linked to stochastic differential utility and certainty equivalents. In these applications, the coefficient of the quadratic term in the generator captures the level of ambiguity aversion and the coefficient of absolute risk aversion, respectively. In the future, we will focus on solving high-dimensional practical problems related to this class of equations.

	\renewcommand{\refname}

\end{document}